\documentclass[a4paper,12pt,reqno,english]{smfart}

\usepackage{aeguill} 
\usepackage{enumerate}
\usepackage{amssymb,amsmath,latexsym,amsthm}
\usepackage[T1]{fontenc}
\usepackage[a4paper, tmargin=1.0in, bmargin=1.0in, lmargin=1.0in,rmargin=1.0in]{geometry}
\usepackage{url}
\usepackage[main=english]{babel}
\usepackage[utf8]{inputenc}
\usepackage{mathrsfs}
\usepackage{xcolor}
\usepackage{bigints}
\usepackage{comment}
\definecolor{violet}{rgb}{0.0,0.2,0.7}
\definecolor{rouge2}{rgb}{0.8,0.0,0.2}
\usepackage{tikz}
\usepackage{empheq}
\usepackage{tikz-cd}
\usetikzlibrary{matrix,arrows,decorations.pathmorphing}
\usepackage{hyperref}
\usepackage{mathpazo} 
\hypersetup{
	bookmarks=true,         
	unicode=false,          
	pdftoolbar=true,        
	pdfmenubar=true,        
	pdffitwindow=false,     
	pdfstartview={FitH},    
	pdftitle={},    
	pdfauthor={},     
	colorlinks=true,       
	linkcolor=rouge2,          
	citecolor=violet,        
	filecolor=black,      
	urlcolor=cyan}           
\usepackage{enumitem}
\usepackage{appendix}

\usepackage{mathtools}
\usepackage{dsfont}

\theoremstyle{plain}
    \newtheorem{thm}{Theorem}[section]
	
	\newtheorem{lem}[thm]{Lemma}
	\newtheorem{prop}[thm]{Proposition}
	
	\newtheorem{cor}[thm]{Corollary}
	
\theoremstyle{plain}
	\newtheorem{bigthm}{Theorem}
	
	\newtheorem{bigcor}[bigthm]{Corollary}
	
    \newtheorem*{bigrmk*}{Remark}

\theoremstyle{definition}
	\newtheorem{defn}[thm]{Definition}
    \newtheorem*{defn*}{Definition}

    \newtheorem*{conj*}{Conjecture}
	\newtheorem*{claim*}{Claim}
    \newtheorem{claim}[thm]{Claim}
	\newtheorem*{ack*}{Acknowledgements}
    \newtheorem*{AID*}{Declaration on the use of AI}
\theoremstyle{remark}
	\newtheorem{rmk}[thm]{Remark}
	\newtheorem*{rmk*}{Remark}

	\newtheorem*{ques*}{Question}
	\newtheorem*{ans*}{Answer}

\numberwithin{equation}{section}
\renewcommand{\labelenumi}{(\roman{enumi})}
\newlist{steps}{enumerate}{1}
\setlist[steps, 1]{label = Step \arabic*:}

\mathcode`l="8000
\begingroup
\makeatletter
\lccode`\~=`\l
\DeclareMathSymbol{\lsb@l}{\mathalpha}{letters}{`l}
\lowercase{\gdef~{\ifnum\the\mathgroup=\m@ne \ell \else \lsb@l \fi}}%
\endgroup

\DeclareFontFamily{U}{MnSymbolC}{}
\DeclareSymbolFont{MnSyC}{U}{MnSymbolC}{m}{n}
\DeclareFontShape{U}{MnSymbolC}{m}{n}{
	<-6>  MnSymbolC5
	<6-7>  MnSymbolC6
	<7-8>  MnSymbolC7
	<8-9>  MnSymbolC8
	<9-10> MnSymbolC9
	<10-12> MnSymbolC10
	<12->   MnSymbolC12}{}
\DeclareMathSymbol{\intprod}{\mathbin}{MnSyC}{'270}
\DeclareMathOperator{\Id}{Id}

\DeclareMathOperator{\im}{im}
\DeclareMathOperator{\tr}{tr}

\DeclareMathOperator{\Vol}{Vol}
\DeclareMathOperator{\Ric}{Ric}

\DeclareMathOperator{\ext}{ext}

\DeclareMathOperator{\Aut}{Aut}

\DeclareMathOperator{\Bisec}{Bisec}
\DeclareMathOperator{\Scal}{Scal}
\DeclareMathOperator{\Lie}{Lie}
\DeclareMathOperator{\red}{red}

\DeclareMathOperator{\Fut}{Fut}
\DeclareMathOperator{\DF}{\mathbf{DF}}

\def\1{\mathds{1}}

\def\E{\mathbf{E}}

\def\J{\mathbf{J}}
\def\L{\mathbf{L}}
\def\M{\mathbf{M}}

\newcommand\rel{{\mathrm{rel}}}
\newcommand\Pol{{\mathrm{P}}}

\newcommand{\ii}{\mathrm{i}}

\newcommand{\homg}{{\widehat{\omega}}}

\newcommand\sm{\sigma}

\newcommand\vph{\varphi}

\newcommand\om{\omega}
\newcommand\ta{\theta}
\newcommand\gm{\gamma}

\newcommand\af{\alpha}
\newcommand\bt{\beta}

\newcommand\Dt{\Delta}
\newcommand\Om{\Omega}

\newcommand\Gm{\Gamma}

\newcommand\RGL{\mathrm{GL}}

\newcommand\BN{\mathbb{N}}

\newcommand\BR{\mathbb{R}}
\newcommand\BC{\mathbb{C}}

\newcommand\BS{\mathbb{S}}

\newcommand\BP{\mathbb{P}}
\newcommand\BT{\mathbb{T}}

\newcommand\T{\mathbb{T}}

\newcommand\CA{\mathcal{A}}

\newcommand\CC{\mathcal{C}}
\newcommand\CD{\mathcal{D}}

\newcommand\CE{\mathcal{E}}

\newcommand\CH{\mathcal{H}}
\newcommand\CI{\mathcal{I}}

\newcommand\CL{\mathcal{L}}

\newcommand\CS{\mathcal{S}}

\newcommand\CX{\mathcal{X}}

\newcommand\lt{\left}
\newcommand\rt{\right}

\newcommand\pl{\partial}
\newcommand\db{\bar{\partial}}

\newcommand\ddb{\partial \bar{\partial}}

\newcommand\dd{\mathrm{d}}
\newcommand\dc{\mathrm{d}^\mathrm{c}}
\newcommand\ddc{\mathrm{d}\mathrm{d}^{\mathrm{c}}}

\newcommand\grad[1]{\nabla #1}

\newcommand\w{\wedge}

\newcommand{\RN}[1]{\textup{\uppercase\expandafter{\romannumeral#1}}}

\makeatletter
\newsavebox{\@brx}
\newcommand{\llangle}[1][]{\savebox{\@brx}{\(\m@th{#1\langle}\)}%
  \mathopen{\copy\@brx\kern-0.5\wd\@brx\usebox{\@brx}}}
\newcommand{\rrangle}[1][]{\savebox{\@brx}{\(\m@th{#1\rangle}\)}%
  \mathclose{\copy\@brx\kern-0.5\wd\@brx\usebox{\@brx}}}
\makeatother

\title[Extremal Sasaki manifolds and weighted K-stability]{Extremal Sasaki manifolds and weighted K-stability}

\author{Simon Jubert}
\address[Simon Jubert]{Dipartimento di Matematica Guido Castelnuovo, Sapienza Universit\`a di Roma; Piazzale Aldo Moro 5, 00185 Roma, Italy \qquad\qquad\qquad\qquad\qquad\qquad\qquad\qquad\qquad}
\email{\href{mailto:simon.jubert@uniroma1.it}{simon.jubert@uniroma1.it} \qquad\qquad\qquad\qquad\qquad\qquad\qquad}
\urladdr{\href{https://sites.google.com/view/simon-jubert/accueil}{https://sites.google.com/view/simon-jubert/accueil}} 

\author{Chung-Ming Pan}
\address[Chung-Ming Pan]{Department of Mathematics, National Tsing Hua University; No.~101, Section~2, Guangfu Road, Hsinchu 300044, Taiwan \qquad\qquad\qquad\qquad\qquad\qquad\qquad\qquad}
\email{\href{mailto:chungmingpan@math.nthu.edu.tw}{chungmingpan@math.nthu.edu.tw} \qquad\qquad\qquad\qquad\qquad\qquad\qquad\qquad\qquad\qquad}
\urladdr{\href{https://chungmingpan.github.io/}{https://chungmingpan.github.io/}}

\date{\today}
\subjclass{53C25, 32Q26, 32Q15}
\keywords{Weighted extremal K\"ahler metrics, extremal Sasaki manifolds, weighted K-stability}

\begin{document} 

\maketitle

\begin{abstract}
We remove the log-concavity assumption on the weight in the analytic weighted Yau--Tian--Donaldson correspondence established in \cite{DiNezza_Jubert_Lahdili_2024, Han_Liu_2024, DiNezza_Jubert_Lahdili_2024_b} and discuss applications to Sasaki and conformally K\"ahler Einstein--Maxwell geometries.
\end{abstract}

\tableofcontents

\section*{Introduction}
Finding canonical K\"ahler metrics in a given K\"ahler class is a central problem in K\"ahler geometry. 
Calabi proposed extremal K\"ahler metrics as natural candidates \cite{Calabi_1982}. 
A K\"ahler metric is extremal if its scalar curvature generates a holomorphic Hamiltonian vector field. 
This notion naturally generalizes both constant scalar curvature K\"ahler (cscK) metrics and K\"ahler--Einstein metrics.

\smallskip
The weighted extremal metric was introduced by Lahdili \cite{Lahdili_2019} to provide a unified framework for several canonical metric problems in K\"ahler (and related) geometry. 
To recall the definition, let $X$ be a compact K\"ahler manifold endowed with a compact torus action $\T \subset \Aut_{\red}(X)$, and let $\af$ be a K\"ahler class. 
After a compatible normalization of the moment maps, all $\BT$-invariant K\"ahler metrics in $\af$ have the same moment polytope, denoted by $\Pol_\alpha$. 
Given a smooth positive weight function $v : \Pol_\af \to \BR_{>0}$, one can define the $v$-weighted scalar curvature $\Scal_v(\om)$ of a $\T$-invariant K\"ahler metric $\om \in \af$. 
For a second positive smooth weight $w : \Pol_\af \to \BR$, the metric $\om$ is called \emph{$(v,w)$-extremal} if 
\begin{equation}{\label{vw:cscK:i}}
    \Scal_v(\om) = (\ell_{v,w}^{\ext} \cdot w)(m_\om),   
\end{equation}
where $\ell_{v,w}^{\ext}$ is the corresponding $(v,w)$-weighted extremal affine-linear function. 

\smallskip
Under the assumption that $v$ is $\log$-concave, the existence of $(v,w)$-extremal metrics has been related to coercivity \cite{DiNezza_Jubert_Lahdili_2024, Han_Liu_2024, DiNezza_Jubert_Lahdili_2024_b} and stability \cite{Boucksom_Jonsson_2025, Han_Liu_2024, Trusiani_2026}. 
Without $\log$-concavity, however, the arguments establishing certain crucial a priori estimates for weighted extremal K\"ahler metrics in \cite{DiNezza_Jubert_Lahdili_2024, Han_Liu_2024} no longer apply.

\smallskip
However, several geometric problems of particular interest lead to \((v,w)\)-weighted extremal equations for weights \(v\) that are not $\log$-concave. A prominent example is the extremal Sasaki problem for a class of Sasaki structures introduced by Apostolov--Calderbank \cite{Apostolov_Calderbank_2021}, which we discuss below in the applications.

\subsection*{Main results}
Our main technical result removes the $\log$-concavity assumption in the integral Laplacian estimates for weighted cscK potentials, which were established by Di~Nezza, the first-named author, and Lahdili \cite{DiNezza_Jubert_Lahdili_2024, DiNezza_Jubert_Lahdili_2024_b}, and independently by Han--Liu \cite{Han_Liu_2024}, under the $\log$-concavity assumption. 

\smallskip
Fix a reference metric $\omega \in \alpha$ and two weight functions $v>0$ and $w$ on $\Pol_\alpha$.

\begin{bigthm}[= Theorem~\ref{thm:main1}]{\label{t:esti}}
Let $\homg \in \alpha$ be a $\T$-invariant $(v,w)$-weighted extremal metric, i.e. solving \eqref{vw:cscK:i}. 
For all $p \geq 1$, there is a constant $C_p > 0$ such that 
\[
    \|\tr_\homg \om\|_{L^{2p+2}(\om^n)} \leq C_p
    \quad\text{and}\quad
    \|\tr_\om \homg\|_{L^{\frac{2p+2}{n-1}}(\om^n)} \leq C_p, 
\]
where $C_p > 0$ is a constant independent of $\homg$. 
\end{bigthm}

Let us stress that our result applies more generally to weighted $(v,w)$-cscK metrics, and not only to the $(v,w)$-extremal case (see Section~\ref{s:csck} for more details).

\smallskip
In previous works \cite{DiNezza_Jubert_Lahdili_2024, Han_Liu_2024, DiNezza_Jubert_Lahdili_2024_b}, the Aubin--Yau inequality is used to obtain integral estimates for the Laplacian of solutions to the weighted cscK equation (or $W^{2,p}$-estimate).
The key point in our approach is that we use the Chern--Lu inequality instead of the standard Aubin--Yau estimate. 
In the unweighted setting, the difference between the two approaches is relatively minor and essentially only in which side one controls the bisectional curvature of the reference metric (see e.g. \cite{Chen_Cheng_2021_1, Pan_To_Trusiani_2023}). 
However, in the weighted case, the two approaches differ in one extra way. 
We shall use $\homg$ to denote a $(v,w)$-extremal metric in $[\om]$ in the following discussion.

\smallskip
Indeed, when one uses the standard Aubin--Yau inequality without modifying the auxiliary function, the term involving the Hessian of $\log v$ appears to have coefficients with linear growth in $\tr_\om \homg$; more precisely, one encounters the following 
\[
    -\frac{1}{\tr_\om \homg} \sum_{i,j} (D_{ij}^2 \log v)(m_\homg) \langle \dd m_\homg^{\xi_i}, \dd m_\homg^{\xi_j} \rangle_\om, 
\]
where $(\xi_i)_i$ is a basis of the Lie algebra $\mathfrak{t} := \Lie(\BT)$. 
Without making additional adjustments to the auxiliary function, this term is difficult to control by the remaining terms in the estimate, and for this reason, the $\log$-concavity assumption on $v$ was imposed in \cite{DiNezza_Jubert_Lahdili_2024, Han_Liu_2024, DiNezza_Jubert_Lahdili_2024_b}, allowing this contribution to be dropped off by its sign. 

\smallskip
By contrast, the corresponding terms in the Chern–Lu approach with the standard auxiliary function have the form
\[
    \frac{1}{\tr_\homg \om} \left(\sum_{i,j} (D_{ij}^2 \log v)(m_\homg) \langle \dd m_\homg^{\xi_i} \w \dc m_\homg^{\xi_j}, \om\rangle_\homg 
    + \sum_i (D_i \log v)(m_\homg) \langle \ddc m_\homg^{\xi_i}, \om\rangle_\homg\right). 
\]
The key observations are that the first pairing $\langle \dd m_\homg^{\xi_i} \w \dc m_\homg^{\xi_j}, \om\rangle_\homg$ is controlled by the fixed background metric (see Claim~\ref{claim:dmdm_ddcm}), and the second part $\langle \ddc m_\homg^{\xi_i}, \om\rangle_\homg$ can be compensated for by other terms arising in the weighted Laplacian of $\tr_\homg \om$. 
This allows us to handle these terms without imposing a sign condition on the Hessian of $\log v$. 
The trade-off is that the subsequent argument establishing the integral estimates for $\tr_\homg \om$ and $\tr_\om \homg$ is more delicate than in the standard Aubin--Yau approach. 

\smallskip
The Chern--Lu inequality also yields a direct Laplacian estimate for weighted K\"ahler--Ricci solitons on Fano manifolds, without first establishing an integral Laplacian estimate. 
Again, our argument does not require $\log$-concavity of $v$, an assumption needed in the usual Aubin--Yau argument used to prove \cite[Proposition~3.8]{Han_Li_2023} (but not included in its statement). 
We give a simple proof of this estimate in Section~\ref{sec:wKRS}.

\begin{rmk*}
Towards the end of this project, we learned from an LLM that the Aubin--Yau method may also yield the desired estimate by adding $\sum_i |m_\homg^{\xi_i}|^2$ to the auxiliary functions. 
Nevertheless, we have decided to present the proof using the Chern--Lu inequality, which we initially developed without AI assistance for weighted K\"ahler--Ricci solitons, requires no extra (mysterious) adjustment to the auxiliary functions, and is also better suited to degenerate situations. 
\end{rmk*}

\subsubsection*{Equivalence between existence and coercivity}
We establish the following equivalence between the existence of weighted extremal K\"ahler metrics and coercivity of the relative weighted Mabuchi functional, removing the assumption of log concavity in \cite{DiNezza_Jubert_Lahdili_2024, Han_Liu_2024, DiNezza_Jubert_Lahdili_2024_b}. 

\begin{bigthm}{\label{t:A}}
Let $(v,w) \in \CC_{>0}^\infty(\Pol_\alpha)^2$ be two weight functions. Then the following are equivalent:
\begin{itemize}
    \item There exists a $(v,w)$-extremal K\"ahler metric in $\alpha$;
    \item The relative weighted Mabuchi functional $\M^{\rel}_{v,w}$ is relatively $\BT^\BC$-coercive, i.e. there exist positive constants $\sigma$ and $C$ such that
    \[
        \M_{v,w}^{\rel}(\homg) \geq \sigma \inf_{\gamma \in \BT^\BC} \J(\gamma^\ast \homg) - C
    \]
for every $\BT$-invariant K\"ahler metric $\homg \in \alpha$.
\end{itemize}
\end{bigthm}

The implication from existence to coercivity was established in
\cite{Apostolov_Jubert_Lahdili_2023}. 
As aforementioned, the converse was proved under the additional assumption that $v$ is log-concave in \cite{Han_Liu_2024, DiNezza_Jubert_Lahdili_2024_b}. 
In this argument, $\log$-concavity is used only to establish integral Laplacian estimates for solutions of the weighted extremal equation and for the twisted equations arising along the continuity path after acting by $\T^{\BC}$. 

\smallskip
Theorem~\ref{t:esti} provides these estimates for the weighted extremal equation without any log-concavity assumption on $v$. 
The corresponding estimates for the twisted equations are proved in Theorem \ref{thm:main2}. 
These results allow the same continuity argument to hold for arbitrary positive weights $v$. 

\subsubsection*{Weighted Yau--Tian--Donaldson correspondence}
Suppose in addition that $(X,\af)$ is polarized, i.e. $\af = 2\pi c_1(L)$ for an ample line bundle $L$ over $X$. 
Very recently, the Yau--Tian--Donaldson (YTD) correspondence has been established, characterizing the existence of a cscK metric in $\af$ in terms of (various) notions of algebro-geometric stability 
\cite{Chi_Li_2022_2, Boucksom_Jonsson_2025, Darvas_Zhang_25, Trusiani_2026}. 
A weighted generalization of this correspondence was proved for $\log$-concave weights $v>0$ in \cite{Boucksom_Jonsson_2025, Trusiani_2026} (see also \cite{Han_Liu_2025}). 
These results rely on the analytic characterizations established in \cite{Apostolov_Jubert_Lahdili_2023, DiNezza_Jubert_Lahdili_2024, Han_Liu_2024, DiNezza_Jubert_Lahdili_2024_b}. 
Since $\log$-concavity is required only in the analytic counterpart, Theorem~\ref{t:A} removes this restriction and yields the weighted YTD correspondence for arbitrary positive weights $v$ on the moment polytope.

\begin{bigcor}{\label{c:B}}
Suppose that $\af = 2 \pi c_1(L)$ for an ample line bundle $L$ over $X$. 
Then the following statements are equivalent.
\begin{itemize}
    \item There exists a $(v,w)$-extremal K\"ahler metric in $\af$.
    \item $(X,\af)$ is $\T$-relatively uniformly $(v,w)$-weighted K-stable; namely, there exists a constant $\lambda > 0$ such that 
    \[
        \DF^{\rel}_{v,w}(\CX,\CA)
        \geq
        \lambda \, \J^{\mathrm{NA}}_{\T^{\BC}}(\CX,\CA),
    \]
    for every $\T$-equivariant dominating smooth test configuration $(\CX,\CA)$.
\end{itemize}
\end{bigcor}

In the above statement, $\DF^{\rel}_{v,w}$ denotes the $\T$-relative $(v,w)$-weighted Donaldson--Futaki invariant and $\J^{\mathrm{NA}}_{\T^{\BC}}$ is the reduced non-Archimedean $\J$-functional. 

\subsubsection*{Applications to extremal Sasaki manifolds}
The existence problem for extremal Sasaki metrics is the odd-dimensional counterpart of Calabi's problem \cite{Boyer_Galicki_Simanca_2008}. 
Consider a $(2n+1)$-dimensional strictly pseudoconvex CR manifold $(N,\CD,J_N)$ of Sasaki type. 
A Sasaki structure $\xi$ on $(N,J_N,\CD)$ determines a \textit{transverse K\"ahler structure} $\om_\xi$, whose scalar curvature $\Scal(\om_\xi)$ is the \emph{Tanaka--Webster scalar curvature}. 
The Sasaki manifold $(N,J_N,\CD,\xi)$ is said to be \emph{extremal} if $\Scal(\om_\xi)$ generates a \emph{transverse Killing vector field} \cite{Boyer_Galicki_2008, Boyer_Galicki_Simanca_2008}.  
The existence of extremal Sasaki metrics has been studied extensively; see \cite{Martelli_Sparks_Yau_2006, Martelli_Sparks_Yau_2008, Legendre_2011, vanCoevering_2013, Boyer_vanCoevering_2018, Collins_Szekelyhidi_2018, Apostolov_Calderbank_2021, Apostolov_Calderbank_Legendre_2021} and the references therein.

\smallskip
When the Sasaki--Reeb vector field $\xi^\CD$ associated with the Sasaki structure $\xi$ generates a circle action $\BS^1_\xi$ on $N$, the Sasaki quotient $M_\xi := N/\BS^1_\xi$ is a K\"ahler orbifold of real dimension $2n$. 
If this action is free on $N$, the Sasaki structure $\xi$ is called \emph{regular}. 
In this case, $M_\xi$ is smooth and inherits a complex structure $J_\xi$. 
We denote by $X_\xi:=(M_\xi,J_\xi)$ the resulting complex manifold, and by abuse of notation, denote its induced K\"ahler form by $\om_\xi$.  
A fundamental feature of this correspondence is that the Sasaki manifold $(N,\CD,J_N,\xi^\CD)$ is extremal if and only if $(X_\xi,\om_\xi)$ is extremal in Calabi's sense. 
By contrast, when $\xi^\CD$ does not generate an $\BS^1$-action, $\xi$ is called \textit{irregular}. 
The Reeb foliation then has a non-Hausdorff leaf space, making the existence problem for extremal Sasaki structures much more delicate. 

\smallskip
In \cite{Apostolov_Calderbank_2021}, Apostolov--Calderbank introduced a construction that allows irregular extremal Sasaki structures to be studied through a regular Sasaki quotient. 
Suppose that a compact torus $\widehat{\T} \subset \Aut_{\mathrm{CR}}(N,\CD,J_N)$ acts on $N$. 
Fix a regular Sasaki structure $\chi$ and a possibly irregular $\xi$ belonging to the Lie algebra of $\widehat{\T}$; in particular, $\chi$ and $\xi$ commute.
Let $(X_\chi,\om_\chi)$ be the regular Sasaki quotient associated with the circle action generated by $\xi^\CD$.

\smallskip
The $\widehat{\T}$-action descends to an action of the torus $\T := \widehat{\T}/\BS^1_\chi \subset \Aut_{\red}(X_\chi)$. 
For every $\T$-invariant $\om_\chi$-relative K\"ahler potential $\vph$ on $X_\chi$, Apostolov--Calderbank \cite[Section~6.2]{Apostolov_Calderbank_2021} associate a deformation $(\CD_\vph,J_\vph,\xi_\vph)$ of the reference Sasaki structure $(\CD,J_N,\xi)$, while keeping the Sasaki--Reeb vector field equal to $\xi^\CD$. 
We denote the resulting family of Sasaki structures by $\CS_\chi(\xi^\CD)$. 

\smallskip
The key feature of this construction is that $(\CD_\vph,J_\vph)$ is Sasaki extremal with Sasaki--Reeb vector field $\xi^\CD$ if and only if
\[
    \om_\vph := \om_\chi+ \ddc_\chi\vph
\]
is a $(v_\xi,w_\xi)$-extremal K\"ahler metric on $(X_\chi,[\om_\chi],\BT)$, in the sense of Lahdili \cite{Lahdili_2019}, for suitable weight functions $v_\xi$ and $w_\xi$. 
Here $\dc_\chi$ denotes the twisted exterior derivative on $X_\chi$. 
Thus, the existence problem for extremal Sasaki structures in $\CS_\chi(\xi^\CD)$ reduces to a weighted extremal problem in the fixed K\"ahler class $[\om_\chi]$. Moreover, we stress that the weight \(v_\xi\) arising in this setting is $\log$-convex. 

\smallskip
Combining Theorem~\ref{t:A}, Corollary~\ref{c:B}, and the aforementioned Apostolov--Calderbank construction \cite{Apostolov_Calderbank_2021}, we obtain the following YTD type correspondence. 
It characterizes the existence of possibly irregular extremal Sasaki structures in terms of the relative uniform weighted K-stability of the regular Sasaki quotient $(X_\chi, [\om_\chi])$.

\begin{bigcor}[= Corollary~\ref{c:sasaki}]
The following statements are equivalent.
\begin{itemize}
    \item There exists an extremal Sasaki structure $(\CD_\vph,J_\vph,\xi_\vph) \in \CS_\chi(\xi^\CD)$.
    \item $(X_\chi,[\om_\chi])$ is $\T$-relatively uniformly $(v_\xi,w_\xi)$-weighted K-stable.
\end{itemize}
\end{bigcor}

\subsection*{Organization of the article}
\begin{itemize}
    \item Section~\ref{s:pre} recaps the framework of weighted cscK geometry, its variational picture, and some useful facts for estimates in the next section.
    \item Section~\ref{s:esticsck} is the most technical part of the article, which contains the proofs of Theorems \ref{t:esti} and \ref{t:A}. 
    \item In Section~\ref{sec:wKRS}, we provide a much simpler proof of Laplacian estimates for weighted K\"ahler--Ricci solitons.
    \item Section~\ref{s:ytd} discusses applications to Sasaki geometry and conformally K\"ahler Einstein--Maxwell geometry, and in particular proves Corollary~\ref{c:sasaki}. 
\end{itemize}

\begin{ack*}
The authors are grateful to Vestislav Apostolov, Eleonora Di~Nezza, Abdellah Lahdili, Eveline Legendre, Tran-Trung Nghiem, Tat Dat T\^o and Antonio Trusiani for helpful discussions. 
The authors also thank S\'ebastien Boucksom and Mattias Jonsson for their interest in this work. 
S.J. is supported by the Ministero dell'Universit\`a e della Ricerca of Italy – Avviso FIS2 grant FIS-2023-0395. 
\end{ack*}

\begin{AID*}
The Chern--Lu approach originated with the authors, who first used it, without AI assistance, to remove the log-concavity assumption for weighted K\"ahler--Ricci solitons (see Claim~\ref{claim:dmdm_ddcm} and Section~\ref{sec:wKRS}). 
The authors shared the weighted K\"ahler--Ricci solitons draft with an LLM, and after several discussions, it provided a strategy for proving Theorem~\ref{thm:main1} and the twisted version. 
The authors then compared this strategy (with AI assistance) with their existing incomplete draft on weighted cscKs (which contains most of the computations in the current version of Section~\ref{sec:wLap_tr} and Section~\ref{sec:auxiliary}, as well as unsuccessful attempts to build the complete proof), simplified the argument in this new strategy, and reworked the estimates to follow more closely the approach used in the proof of \cite[Proposition~5.6]{Pan_To_Trusiani_2023}, which had also guided the authors' initial attempts in the beginning of the project. 
The main mathematical input from the LLM was a strategy for analyzing the term $\CI$ in Section~\ref{eq:analyze_I}, in particular the term $I_1$, and the introduction of an original form of an auxiliary function $H$ (slightly different from the current version) that simplifies the computations. 
An LLM was also used to polish the English and suggest improvements to the exposition. 
The authors wrote the entire manuscript and take full responsibility for its content and correctness. 
\end{AID*}

\section{Preliminaries}{\label{s:pre}}
In this article, we use the normalization $\dc = \ii (\db - \pl)$ and $\ddc = 2\ii \ddb$. 
Denote by 
\[
    \tr_\om \ta = \frac{n \ta \w \om^{n-1}}{\om^n}, 
    \quad
    \langle \dd f, \dd g \rangle_\om = \frac{n \dd f \w \dc g \w \om^{n-1}}{\om^n}, 
\]
\[
    \Dt_\om f = \tr_\om (\ddc f) = \frac{n \ddc f \w \om^{n-1}}{\om^n},
    \quad
    |\dd f|_\om^2 = \langle \dd f, \dd f \rangle_\om = \frac{n \dd f \w \dc f \w \om^{n-1}}{\om^n}.
\]
We also take $\Ric(\om) := -\ddc \log \om^n$ \footnote{We remark that the Ricci form here is twice the \emph{usual} Ricci form in the literature.}. 

\subsection{Weighted constant scalar curvature K\"ahler metrics}{\label{s:csck}}
Let $X$ be an $n$-dimensional compact K\"ahler manifold, let $\Aut_{\red}(X)$ denote its reduced automorphism group, and fix a K\"ahler class $\alpha \in H^2(X,\BR)$. 
Let $\T\subset\Aut_{\red}(X)$ be a real torus, and choose a $\T$-invariant K\"ahler metric $\omega \in \alpha$. 
It is well known (see e.g. \cite{Gauduchon_book}) that the $\T$-action on $(X,\omega)$ is Hamiltonian, and we denote the corresponding moment map by $m_\om : X \to \mathfrak{t}^\vee$, where $\mathfrak{t}^\vee$ is the dual of the Lie algebra $\mathfrak{t}$ of $\T$. 
By the Atiyah--Guillemin--Sternberg convexity theorem \cite{Atiyah_82, Guillemin_Sternberg_82}, its image is a compact convex polytope $\Pol_\om = \im(m_\om) \subset \mathfrak{t}^\vee$.

\smallskip
We consider the space of $\T$-invariant $\omega$-relative K\"ahler potentials
\[
    \CH(X,\omega)^\T 
    := \left\{\vph \in \CC^\infty(X,\BR)^\BT \mid \om_\vph = \om + \ddc \vph > 0\right\}.
\]
For each $\vph \in \CH(X,\omega)^\T$, we normalize the moment map of ${\omega_\varphi}$ by requiring
\begin{equation}{\label{norm:moment}}
    m_{\omega_\varphi} = m_\omega+ \dc\varphi.
\end{equation}
By \cite[Lemma~1]{Lahdili_2019}, this normalization ensures that the image of $m_{\omega_\varphi}$ is independent of $\varphi \in \CH(X,\omega)^\T$. 
We therefore denote the common image by $\Pol_\alpha$ and refer to it as the \textit{moment polytope} of $(X,\alpha,\T)$. 
In the sequel, we assume this normalization throughout.

\smallskip
Lahdili \cite{Lahdili_2019} introduced weighted cscK metrics as a framework unifying all the classical notions of canonical K\"ahler metrics. 
We shall give a quick overview of Lahdili's package below. 

\smallskip
For simplicity, write $\homg:=\omega_\varphi$. 
Given a positive weight $v\in \CC^\infty_{>0}(\Pol_\alpha)$, we define the \textit{$v$-weighted Ricci form} by
\begin{equation*}
    \Ric_v(\homg)
    := \Ric(\homg) - \ddc \log v(m_\homg).
\end{equation*}
For a $\BT$-invariant closed $(1,1)$-form $\bt$ equipped with a moment map $m_\bt$, its \emph{$v$-weighted trace} with respect to $\homg$ is 
\[
    \tr_{\homg,v}(\beta)
    := \tr_{\homg}(\beta)
    + \left\langle (D \log v)(m_{\homg}),m_\beta\right\rangle.
\]
Here $(D \log v)(m_{\homg})$ denotes the differential of $\log v$ on $\mathfrak{t}^\vee$ evaluated at $m_\homg$, and the pairing is the natural one between $\mathfrak{t}$ and $\mathfrak{t}^\vee$. 
The \emph{$v$-weighted scalar curvature} of a $\T$-invariant K\"ahler metric $\homg$ is then defined by
\begin{equation*}
    \Scal_v(\homg)
    :=v(m_{\homg})
     \mathrm{tr}_{\homg,v}\bigl(\Ric_v(\homg)\bigr),
\end{equation*}
with the natural moment map normalization for the weighted Ricci form.
We also have the \emph{$v$-weighted Laplacian} with respect to $\homg$ given as follows   
\[
    \Dt_{\homg,v} f 
    = \tr_{\homg,v} (\ddc f) 
    = \Dt_\homg f + \langle \dd \log v(m_\homg), \dd f\rangle_\homg. 
\]

\smallskip
The weighted scalar curvature also admits a moment map interpretation extending the classical Donaldson--Fujiki picture \cite{Fujiki_1992, Donaldson_1997}.
More precisely, it arises from the $\BT$-equivariant Hamiltonian diffeomorphisms on the space of $\BT$-invariant compatible almost complex structures (see e.g. \cite{Lahdili_2019}). 

\smallskip
Given a second weight $w \in \CC^\infty(\Pol_\alpha)$, a $\T$-invariant K\"ahler metric $\homg$ is called \emph{$(v,w)$-weighted cscK} if 
\begin{equation}{\label{wcsck}}
    \Scal_v(\homg)=w(m_{\homg}).
\end{equation}
The relation between \eqref{wcsck} and various geometric problems is discussed extensively in \cite{Lahdili_2019}. 
In this article, an important class of examples arises from the extremal Sasaki structures (see Section~\ref{s:sasaki}). 

\smallskip
We next introduce the notion of weighted extremal K\"ahler metrics, following \cite{Lahdili_2019}, as a weighted generalization of Calabi's original notion \cite{Calabi_1982}. 
Assume from now on that $w > 0$ on $\Pol_\af$. 
There is a unique affine-linear function $\ell^{\mathrm{ext}}_{v,w}$ characterized by the vanishing of the weighted Futaki invariant $\Fut_{v,w}$: 
\begin{equation}{\label{lext}}
    \Fut_{v,w}(\ell)
    := \int_X
    \bigl( \Scal_v(\omega) - w(m_\omega) \ell^{\mathrm{ext}}_{v,w}(m_\omega) \bigr) \ell(m_\omega)\,
    \omega^n = 0, \quad \forall \ell \in \mathrm{Aff}(\mathfrak{t}^\vee).
\end{equation}
A $\T$-invariant K\"ahler metric $\homg \in \af$ is then called \emph{$(v,w)$-extremal} if it satisfies
\[
    \Scal_v(\homg)
    = (w\ell^{\mathrm{ext}}_{v,w})(m_{\homg}).
\]
Thus, $(v,w)$-extremal metrics are precisely $(v,w\ell^{\mathrm{ext}}_{v,w})$-cscK metrics, for a weight $w>0$. 
In particular, when $v = w = 1$, this reduces to the classical notion of Calabi extremal K\"ahler metrics with the extremal vector field lying in $\mathfrak{t}$. 

\smallskip
An important feature of weighted extremal metrics is their uniqueness up to automorphism. 
More precisely, $\BT$-invariant $(v,w)$-extremal metrics in $\af$ are unique up to the identity component of the centralizer of $\BT$ in $\Aut_{\red}(X)$. 
In particular, if $\T$ is maximal, they are unique up to the complexified action $\BT^\BC$ (cf. \cite{Apostolov_Jubert_Lahdili_2023, Lahdili_2023}).

\subsection{Variational picture}
The existence of extremal metrics admits a variational formulation. 
In the unweighted setting, the relative Mabuchi functional \cite{Mabuchi_1987} plays an important role in connecting the existence of the extremal metric and an algebraic notion of stability.
The relative Mabuchi functional $\M^{\rel}$, defined on $\CH(X,\om)^\T$, is an Euler--Lagrange functional for 
\[
    \CH(X,\om)^\T \ni \vph \mapsto \left(\ell^{\mathrm{ext}}(m_{\om_\vph})-\Scal(\om_\vph)\right) \om_\vph^n, 
\]
where $\ell^{\mathrm{ext}}:=\ell^{\mathrm{ext}}_{1,1}$ is the \textit{unweighted} extremal affine-linear function \eqref{lext}.

\smallskip
A series of works by many experts \cite{Darvas_Rubinstein_2017, Boucksom_Hisamoto_Jonsson_2019, He_2019, Berman_Darvas_Lu_2020, Chen_Cheng_2021_2, Chi_Li_2022_2, Boucksom_Jonsson_2025, Darvas_Zhang_25, Trusiani_2026} has characterized the existence of an extremal metric on polarized manifolds in terms of relative uniform K-stability. 
A key analytic ingredient in this correspondence is the equivalence between the existence of an extremal K\"ahler metric and coercivity of the relative Mabuchi functional modulo automorphisms. 

\smallskip
To state the coercivity condition precisely, we recall the related energy functionals. 
For $\varphi \in \CH(X,\om)^\T$, the energy functional is defined by 
\[
    \E(\varphi) := \frac{1}{(n+1)\Vol(X,\om)} \sum_{j=0}^n \int_X \varphi \om_\varphi^j \w \om^{n-j},
\]
where $\Vol(X,\om) = \int_X \om^n$. 
It is a primitive of the normalized complex Monge--Amp\`ere operator. 
The associated $\J$-functional is given by  
\[
    \J(\varphi) = \frac{1}{\Vol(X,\om)} \int_X \varphi \om^n - \E(\varphi) \geq 0.
\]
These functionals provide a way to measure ``distance'' between two potentials in $\CH(X,\om)^\T$. 
Both $\M^{\rel}$ and $\J$ are preserved up to adding constants; therefore, they can be regarded as functionals defined over the space of $\T$-invariant K\"ahler metrics within $[\om]$.

\smallskip
We say that a functional $\mathbf{F}$ defined on the space of $\T$-invariant K\"ahler metrics in $\alpha$ is \emph{relatively $\T^{\BC}$-coercive} if there exist constants $\sm, C > 0$ such that  
\[
    \mathbf{F}(\homg) \geq \sm \inf_{\gm \in \T^\BC} \J(\gm^\ast \homg) - C
\]
for every $\T$-invariant K\"ahler metric $\homg \in [\om]$. 

\smallskip
The weighted cscK problem admits a similar variational formulation. 
Using the notation $\Scal_v(\homg) = v(m_\homg) \tr_{\homg,v}(\Ric_v(\homg))$ introduced before, the weighted Mabuchi functional $\M_{v,w}$ introduced in \cite{Lahdili_2019} is characterized by 
\[
    \frac{\dd}{\dd t}_{|t=0} \M_{v,w}(\vph + tf) = \int_X f (w(m_{\om_\vph}) - \Scal_v(\om_\vph)) v(m_{\om_\vph}) \om_\vph^n
    \quad\text{and}\quad
    \M_{v,w}(0) = 0, 
\]
for every smooth $\BT$-invariant function $f$ and $\vph \in \CH(X,\om)^\BT$. Assuming that the weights $(v,w)$ satisfy the normalization condition
\begin{equation}{\label{norma:weigh}}
    \int_X \bigl(\Scal_v(\om) - w(m_\om)\bigr)\om^n = 0,
\end{equation}
the functional $\M_{v,w}$ is invariant under the addition of constants and therefore descends to a functional on the space of $\T$-invariant K\"ahler metrics. We note that the above condition is cohomological by \cite[Lemma~2]{Lahdili_2019}, and is necessary for the existence of a $(v,w)$-cscK metric in $\alpha$. We also denote by 
\[
    \M_{v,w}^{\rel} 
    := \M_{v,w\ell_{v,w}^{\ext}}
\]
the relative weighted Mabuchi energy.

\smallskip
In the works \cite{Apostolov_Jubert_Lahdili_2023, DiNezza_Jubert_Lahdili_2024, Han_Liu_2024, DiNezza_Jubert_Lahdili_2024_b, Boucksom_Jonsson_2025, Trusiani_2026}, 
the authors establish the corresponding existence-coercivity and stability results under the additional assumption that $v$ is $\log$-concave. 
The $\log$-concavity assumption enters in the implication from relative $\T^\BC$-coercivity of $\M_{v,w}$ to the existence of weighted cscKs. 
More precisely, the assumption is needed for establishing a uniform integral Laplacian estimate along the continuity path towards the weighted cscK metrics.  
Removing this assumption from these estimates is the main analytic objective of the present work. 

\subsection{Useful facts}
We then recall some lemmas that are useful for removing the $\log$-concavity assumption from the estimates. 
The Chern--Lu inequality (cf. \cite[Proposition~7.2]{Rubinstein_2014}) will be one of the key ingredients distinguishing our proof from earlier works \cite{DiNezza_Jubert_Lahdili_2024, Han_Liu_2024}. 

\begin{lem}[Chern--Lu inequality]\label{lem:Chern-Lu}
Let $X$ be a complex manifold equipped with two K\"ahler metrics $\om$ and $\homg$. 
Assume that there is a constant $B \in \BR$ such that $\Bisec(\om) \leq B$. 
Then 
\[
    \Dt_\homg \log \tr_\homg \om
    \geq \frac{\langle \Ric(\homg) , \om \rangle_\homg}{\tr_\homg \om} - B \tr_\homg \om.
\]
\end{lem}

We shall also use the following standard equality (see e.g. \cite[Lemma~4.7]{Szekelyhidi_book}) and the integration by parts formula of the weighted Laplacian. 

\begin{lem}\label{lem:tensor_eq_Sze}
Let $\af$, $\bt$ be two $(1,1)$-forms and let $\om$ be a hermitian metric. 
Then
\[
    n(n-1) \af \w \bt \w \om^{n-2} 
    = [(\tr_\om \af)(\tr_\om \bt) - \langle \af,\bt \rangle_\om] \om^n.
\]
In particular, if $\af \geq 0$ and $\bt \geq 0$, 
\[
    n (n-1) \af \w \bt \w \om^{n-2} 
    \leq (\tr_\om \af) (\tr_\om \bt) \om^n.
\]
\end{lem}

We next recall an integration by parts formula for the weighted Laplacian from \cite[Appendix~A]{DiNezza_Jubert_Lahdili_2024}: 
\begin{lem}\label{lem:IPP_wLap}
For all smooth functions $f$ and $g$,
\[
    \int_X g (\Dt_{\om,v} f) v(m_\om) \om^n 
    = - \int_X \langle \dd f, \dd g\rangle_\om v(m_\om) \om^n
    = \int_X f (\Dt_{\om,v} g) v(m_\om) \om^n,
\]
where $\Dt_{\om,v} f = \tr_{\om,v} (\ddc f) = \Dt_\om f + \langle \dd \log v(m_\om), \dd f \rangle_\om$ is the weighted Laplacian.
In particular, we have 
\[
    \int_X (\Dt_{\om,v} f) v(m_\om) \om^n = 0.
\]
\end{lem}

\section{Integral estimates of Laplacian for weighted cscK equations}{\label{s:esticsck}}
In this section, we shall give a proof of integral estimates for the Laplacian of $(v,w)$-cscK equations and their twists along the continuity path in \cite{DiNezza_Jubert_Lahdili_2024_b}. 
Let us stress that our proof does not require $v$ to be ``$\log$-concave'', which is needed in the previous works \cite{DiNezza_Jubert_Lahdili_2024, Han_Liu_2024, DiNezza_Jubert_Lahdili_2024_b} and removing this assumption fills an important missing part in recent progress of the weighted uniform Yau--Tian--Donaldson correspondence \cite{Apostolov_Jubert_Lahdili_2023, Han_Liu_2024, DiNezza_Jubert_Lahdili_2024_b, Boucksom_Jonsson_2025, Trusiani_2026}.

\smallskip
Let $X$ be an $n$-dimensional compact K\"ahler manifold, let $\BT \subset \Aut_{\red}(X)$ be a real torus, and fix a $\BT$-invariant K\"ahler metric $\om$ in a given K\"ahler class $\af$.  
Consider two weight functions $v \in \CC_{>0}^\infty(\Pol_\af)$ and $w \in \CC^\infty(\Pol_\af)$ defined on the moment polytope $\Pol_\af = \im(m_\om) \subset \mathfrak{t}^\vee$.
The $(v,w)$-cscK equation \eqref{wcsck} is equivalent to
\begin{equation}\label{eq:wcscK}\tag{wcscK}
\begin{cases}
    v(m_{\om_\vph}) \om_\vph^n = e^F \om^n, \\
    \Dt_{\om_\vph,v} F = -w(m_{\om_\vph}) + \tr_{\om_\vph,v} \Ric(\om), 
\end{cases}
\end{equation}
where $\vph \in \CH(X,\om)^\BT$ and $F$ is a smooth $\BT$-invariant function.
We shall denote by $\homg = \om_\vph$ for simplicity. 
In the sequel, we fix a basis $(\xi_i)_{i=1,\cdots,r}$ of the Lie algebra $\mathfrak{t}$ of $\BT$, where $r = \dim \BT$. 

\smallskip
This section aims to prove the following integral estimate for the Laplacian of $\vph$:

\begin{thm}\label{thm:main1}
Let $(\vph, F)$ be a $\BT$-invariant pair solving \eqref{eq:wcscK}. 
For all $p \geq 1$, there is a constant $C_p > 0$ such that 
\[
    \|\tr_\homg \om\|_{L^{2p+2}(\om^n)} \leq C_p
    \quad\text{and}\quad
    \|\tr_\om \homg\|_{L^{\frac{2p+2}{n-1}}(\om^n)} \leq C_p,
\]
where $C_p$ depends only on $n$, $p$, $\om$, $\|\vph\|_{L^\infty}$, $\|F\|_{L^\infty}$, 
$\displaystyle \max_{\Pol_\af} \left(|v| + \sum_i |D_i \log v| + \sum_{i,j} |D_{ij}^2 \log v|\right)$, 
$\displaystyle \min_{\Pol_\af} v$, 
$\displaystyle \max_{\Pol_\af} |w|$, 
$\displaystyle \max_i \max_{\eta \in \Pol_\af}|\langle\xi_i,\eta\rangle_{\mathfrak{t},\mathfrak{t}^\vee}|$, and 
$\displaystyle \max_i \max_{x \in X}|m_{\Ric(\om)}^{\xi_i}(x)|$. 
\end{thm}

\subsection{Proof of Theorem~\ref{thm:main1}}\label{sec:pf_main1}
In the sequel, we shall use the following constants: 
\begin{itemize}
    \item $A_0 > 0$ such that $-A_0 \om \leq \Ric(\om) \leq A_0 \om$; 
    \item $A_1 > 0$ such that $-A_1 \om \leq \CL_{J\xi_i} \om \leq A_1 \om$ for all $\xi_i$; 
    \item $A_2 > 0$ such that $|g_\om(\xi_i,\xi_j)| = |\om(\xi_i,J\xi_j)| \leq A_2$ for all $\xi_i, \xi_j$;
    \item $B > 0$ such that $\Bisec(\om) \leq B$; 
    \item $C_v, C_w > 0$ such that 
    \[
        \max_{\Pol_\af} \left(|v| + \sum_i |D_i \log v| + \sum_{i,j} |D_{ij}^2 \log v|\right) \leq C_v
        \quad\text{and}\quad
        \max_{\Pol_\af} |w| \leq C_w;
    \]
    \item $C_0 > 0$ such that 
    \[
        \|\vph\|_{L^\infty} + \|F\|_{L^\infty} \leq C_0;
    \]
    \item $C_\Pol > 0$ such that $|\langle \xi_i, \eta \rangle_{\mathfrak{t},\mathfrak{t}^\vee}| \leq C_\Pol$ for all $\eta \in \Pol_\af \subset \mathfrak{t}^\vee$ and all $\xi_i$;
    \item $C_{\Ric} > 0$ such that $|m_{\Ric(\om)}^{\xi_i}| \leq C_{\Ric}$ for all $\xi_i$. 
\end{itemize}
We first make the following useful claim, which will also be helpful in studying weighted K\"ahler--Ricci solitons (see Section~\ref{sec:wKRS}). 

\begin{claim}\label{claim:dmdm_ddcm}
Let $\om$ and $\homg$ be two $\BT$-invariant K\"ahler metrics. 
We have
\[
    \langle \dd m_\homg^{\xi_i} \w \dc m_\homg^{\xi_j} , \om\rangle_\homg = g_\om (\xi_i,\xi_j)
    \quad\text{and}\quad 
    \langle \ddc m_\homg^{\xi_i}, \om \rangle_\homg 
    = \CL_{J\xi_i} \tr_\homg \om - \tr_\homg(\CL_{J\xi_i} \om).
\]
\end{claim}
\begin{proof}
The first equality comes directly from the definition that $\dd m_\homg^{\xi_i} = - \homg(\xi_i,\cdot)$.
To show the second one, we first note that, by Cartan's formula, we have  
\[
    \ddc m_\homg^{\xi_i} = - \dd \homg(J\xi_i,\cdot) = - \CL_{J\xi_i} \homg.
\]
By Lemma~\ref{lem:tensor_eq_Sze}, we have 
\begin{align*}
    \langle \ddc m_\homg^{\xi_i}, \om \rangle_\homg 
    &= (\Dt_\homg m_\homg^{\xi_i})(\tr_\homg \om) - n(n-1) \frac{\ddc m_\homg^{\xi_i} \w \om \w \homg^{n-2}}{\homg^n} \\
    &= (\Dt_\homg m_\homg^{\xi_i})(\tr_\homg \om) + n(n-1) \frac{(\CL_{J\xi_i}\homg) \w \om \w \homg^{n-2}}{\homg^n}.
\end{align*}
Note that
{\small
\begin{align*}
    &\CL_{J\xi_i} \tr_\homg \om 
    =n \CL_{J\xi_i} \lt(\frac{\om \w \homg^{n-1}}{\homg^n}\rt)\\
    &\qquad= n\frac{(\CL_{J\xi_i} \om) \w \homg^{n-1}}{\homg^n} + n(n-1) \frac{(\CL_{J\xi_i} \homg) \w \om \w \homg^{n-2}}{\homg^n} 
    - n^2 \frac{\om \w \homg^{n-1}}{\homg^n} \frac{(\CL_{J\xi_i} \homg) \w \homg^{n-1}}{\homg^n}\\
    &\qquad= n\frac{(\CL_{J\xi_i} \om) \w \homg^{n-1}}{\homg^n} + n(n-1) \frac{(\CL_{J\xi_i} \homg) \w \om \w \homg^{n-2}}{\homg^n} 
    + (\tr_\homg \om) (\Dt_\homg m_\homg^{\xi_i}).
\end{align*}
}%
Hence, 
\begin{equation}
    \langle \ddc m_\homg^{\xi_i}, \om \rangle_\homg
    = \CL_{J\xi_i} \tr_\homg \om - \tr_\homg(\CL_{J\xi_i} \om), 
\end{equation}
and we complete the proof.
\end{proof}

\subsubsection{Weighted Laplacian of $\tr_\homg \om$}\label{sec:wLap_tr}
By the Chern--Lu inequality (Lemma~\ref{lem:Chern-Lu}) and direct computations, we obtain the following estimates: 
\begin{equation}\label{eq:ess1}
\begin{split}
    \Dt_{\homg,v} \log \tr_\homg \om 
    &= \Dt_\homg \log \tr_\homg \om 
    + \langle \dd \log v(m_\homg), \dd \log \tr_\homg \om \rangle_\homg \\
    &\geq \frac{\langle \Ric(\homg), \om\rangle_\homg}{\tr_\homg \om} - B \tr_\homg \om 
    + \langle \dd \log v(m_\homg), \dd \log \tr_\homg \om \rangle_\homg \\
    &= \frac{\langle \Ric(\om), \om\rangle_\homg}{\tr_\homg \om} - \frac{\langle \ddc F, \om\rangle_\homg}{\tr_\homg \om} + \frac{\langle \ddc \log v(m_\homg), \om\rangle_\homg}{\tr_\homg \om} - B \tr_\homg \om \\
    &\qquad + \langle \dd \log v(m_\homg), \dd \log \tr_\homg \om \rangle_\homg \\
    &\geq - (A_0+B) \tr_\homg \om - \frac{\langle \ddc F, \om \rangle_\homg}{\tr_\homg \om} \\
    &\qquad + \frac{1}{\tr_\homg \om} \left( \langle \ddc \log v(m_\homg), \om \rangle_\homg + \langle \dd \log v(m_\homg), \dd \tr_\homg \om \rangle_\homg\right).
\end{split}
\end{equation}
The third line comes from $\Ric(\homg) = \Ric(\om) - \ddc F + \ddc \log v(m_\homg)$ and the last inequality uses $\Ric(\om) \geq -A_0 \om$. 
Note that 
\[
    \ddc \log v(m_\homg) 
    = \sum_{i,j} (D^2_{ij} \log v)(m_\homg) \dd m_\homg^{\xi_i} \w \dc m_\homg^{\xi_j} 
    + \sum_i (D_i \log v)(m_\homg) \ddc m_\homg^{\xi_i}.
\]
Applying Claim~\ref{claim:dmdm_ddcm}, we infer 
\begin{equation}\label{eq:wKRS1}
\begin{split}
    \langle \ddc \log v(m_\homg), \om \rangle_\homg 
    &= \sum_{i,j} (D_{ij}^2 \log v)(m_\homg) g_\om(\xi_i,\xi_j)\\
    &\qquad + \sum_i (D_i \log v)(m_\homg) (\CL_{J\xi_i} \tr_\homg \om - \tr_\homg(\CL_{J\xi_i} \om)).
\end{split}
\end{equation}
We also have
\begin{align*}
    \langle \dd \log v(m_\homg), \dd \tr_\homg \om \rangle_\homg 
    = - \sum_i (D_i \log v)(m_\homg) (\CL_{J \xi_i} \tr_\homg \om).
\end{align*}
Thus, we derive the following estimates for the last two terms in \eqref{eq:ess1}: 
\begin{equation}\label{eq:ess2}
\begin{split}
    &\langle \ddc \log v(m_\homg), \om \rangle_\homg + \langle \dd \log v(m_\homg), \dd \tr_\homg \om \rangle_\homg \\
    &\qquad = \sum_{i,j} (D_{ij}^2 \log v)(m_\homg) g_\om(\xi_i,\xi_j) - \sum_i (D_i \log v)(m_\homg)  \tr_\homg (\CL_{J \xi_i} \om)\\
    &\qquad \geq - C_v A_2 - C_v A_1 \tr_\homg \om. 
\end{split}
\end{equation}
Putting \eqref{eq:ess1} and \eqref{eq:ess2} together, we get  
\begin{equation}\label{eq:ess2.5}
    \Dt_{\homg,v} \log \tr_\homg \om 
    \geq - (A_0+B) \tr_\homg \om - \frac{\langle \ddc F, \om \rangle_\homg}{\tr_\homg \om} - C_v A_1 - C_v A_2 \frac{1}{\tr_\homg \om}. 
\end{equation}

\subsubsection{Auxiliary functions}\label{sec:auxiliary}
Following \cite[Proof of Proposition~5.6]{Pan_To_Trusiani_2023}, we now introduce the auxiliary functions $G$ and $u$ that help to establish the estimates, as well as an auxiliary function $H$. 
We define
\[
    G = -(aF + b\vph), \quad
    u = e^G \tr_\homg \om, \quad\text{and}\quad
    H = e^G u^{2p} = \frac{u^{2p+1}}{\tr_\homg \om}.
\] 
We first note that, by the AM-GM inequality, 
\[
    \tr_\homg \om \geq n \sqrt[n]{\frac{\om^n}{\homg^n}} = n \sqrt[n]{\frac{v(m_\homg)}{e^F}} \geq c_0, 
\]
where $c_0 = c_0(\min_{\Pol_{[\om]}}v, C_0) > 0$ is a uniform constant. 
Recall that from \eqref{eq:ess2.5}, the above estimate, and Lemma~\ref{lem:tensor_eq_Sze}, we have 
\begin{align*}
    \Dt_{\homg,v} \log \tr_\homg \om 
    &\geq - \frac{\langle \ddc F, \om \rangle_\homg}{\tr_\homg \om} - (A_0+B) \tr_\homg \om  - C_v A_1 - C_v A_2 \frac{1}{\tr_\homg \om} \\
    &\geq - \frac{\langle \ddc F, \om \rangle_\homg}{\tr_\homg \om} - C_1 (1 + \tr_\homg \om) \\
    &= - \Dt_\homg F + \frac{\Dt_\Om F}{\tr_\homg \om} - C_1(1 + \tr_\homg \om) \\
    &= - \Dt_{\homg, v} F + \langle \dd \log v(m_\homg), \dd F\rangle_\homg + \frac{\Dt_\Om F}{\tr_\homg \om} - C_1(1 + \tr_\homg \om),
\end{align*}
where $C_1 = C_1(A_0, A_1, A_2, B, C_v, c_0) > 0$ is a uniform constant and $\Dt_\Om$ is a linear second order elliptic operator defined by 
\[
    \Dt_\Om f := \frac{\ddc f \w \Om}{\homg^n}
    \quad\text{and}\quad
    \Om := n(n-1) \om \w \homg^{n-2}.
\]
Then we derive that 
\begin{equation}\label{eq:wLap_logu_1}
{\small
\begin{split}
    &\Dt_{\homg, v} \log u 
    = -\Dt_{\homg, v} (aF + b\vph) + \Dt_{\homg,v} \log \tr_\homg \om\\
    &\qquad\geq -(a+1) \Dt_{\homg,v} F - b \tr_{\homg,v} (\homg -\om) + \langle\dd \log v(m_\homg), \dd F \rangle_\homg + \frac{\Dt_\Om F}{\tr_\homg \om} - C_1(1 + \tr_\homg \om)\\
    &\qquad= (a+1) (w(m_\homg) - \tr_{\homg,v}\Ric(\om)) - b \tr_{\homg,v} (\homg -\om) \\
    &\qquad\qquad + \langle\dd \log v(m_\homg), \dd F \rangle_\homg + \frac{\Dt_\Om F}{\tr_\homg \om} - C_1(1 + \tr_\homg \om).
\end{split}}%
\end{equation}
Note that 
\begin{equation}\label{eq:wtr_homg_om}
\begin{split}
    \tr_{\homg,v} (\homg - \om) 
    &= \tr_\homg(\homg - \om) + \sum_i (D_i \log v)(m_\homg)(m_\homg^{\xi_i} - m_\om^{\xi_i}) \\
    &\leq n - \tr_\homg \om + 2 C_v C_\Pol. 
\end{split}
\end{equation}
Putting \eqref{eq:wLap_logu_1} and \eqref{eq:wtr_homg_om} together, we gain 
\begin{equation}\label{eq:wLap_logu}
    \Dt_{\homg,v} \log u 
    \geq (b - (a+1)A_0 - C_1) \tr_\homg \om - C_{2,a,b} + \frac{\Dt_\Om F}{\tr_\homg \om} + \langle \dd \log v(m_\homg), \dd F\rangle_\homg, 
\end{equation}
where $C_{2,a,b}$ depends on $a$, $b$, $C_w$, $C_{\Ric}$, $C_v$, $C_\Pol$, $C_1$ and $n$. 

\smallskip
By direct computations, we have 
\begin{align*}
    \frac{1}{2p+1} \Dt_{\homg,v} u^{2p+1} 
    &= u^{2p} \Dt_{\homg,v} u + 2pu^{2p-1} |\dd u|_\homg^2 \\
    &= u^{2p} \Dt_{\homg,v} u + 2p u^{2p+1} |\dd \log u|_\homg^2
\end{align*}
and 
\[
    \Dt_{\homg,v} u \geq u \Dt_{\homg,v} \log u.
\]
Thus, using Lemma~\ref{lem:IPP_wLap}, we derive that 
\begin{equation}\label{eq:IPP1}
\begin{split}
    0 &= \int_X \frac{1}{2p+1} \Dt_{\homg,v} u^{2p+1} v(m_\homg) \homg^n \\
    &\geq 2p \int_X u^{2p+1} |\dd \log u|_\homg^2 v(m_\homg) \homg^n 
    + \int_X u^{2p+1} (\Dt_{\homg,v} \log u) v(m_\homg) \homg^n
\end{split}
\end{equation}
Combining \eqref{eq:wLap_logu} and \eqref{eq:IPP1}, we obtain that 
\begin{equation}\label{eq:IPP2}
\begin{split}
    0 &\geq 2p \int_X u^{2p+1} |\dd \log u|_\homg^2 v(m_\homg) \homg^n \\
    &\qquad + (b - (a+1) A_0 - C_1) \int_X u^{2p+1} (\tr_\homg \om) v(m_\homg) \homg^n\\
    &\qquad - C_{2,a,b} \int_X u^{2p+1} v(m_\homg) \homg^n\\
    &\qquad + \underbrace{\int_X H \left[\Dt_\Om F + \tr_\homg\om \langle \dd \log v(m_\homg), \dd F\rangle_\homg \right] v(m_\homg) \homg^n}_{:= \CI}.    
\end{split}
\end{equation}

\subsubsection{Analyze the term $\CI$.}\label{eq:analyze_I}
Since $F = -\frac{1}{a}(G+b\vph)$, we have 
\begin{equation}\label{eq:mix1}
\begin{split}
    \CI &= \underbrace{- \frac{1}{a} \int_X H \left[\Dt_\Om G + \tr_\homg \om \langle \dd \log v(m_\homg), \dd G\rangle_\homg\right] v(m_\homg) \homg^n}_{=: I_1}\\
    &\qquad \underbrace{- \frac{b}{a} \int_X H \left[\Dt_\Om \vph + \tr_\homg \om \langle \dd \log v(m_\homg), \dd \vph\rangle_\homg\right] v(m_\homg) \homg^n}_{=: I_2}. 
\end{split}
\end{equation}
We shall analyze $I_1$ and $I_2$ separately. 

\medskip
\noindent{\bf Analyzing $I_2$:}
We first study the term $I_2$. 
Note that 
\begin{equation}\label{eq:P_est}
\begin{split}
    \Dt_\Om \vph
    &= n(n-1) \frac{(\homg - \om) \w \om \w \homg^{n-2}}{\homg^n} \\
    &\leq n(n-1) \frac{\om \w \homg^{n-1}}{\homg^n} 
    = (n-1) \tr_\homg \om. 
\end{split}
\end{equation}
On the other hand, 
\begingroup
\allowdisplaybreaks
\begin{align*}
    \langle \dd \log v(m_\homg),\dd \vph\rangle_\homg 
    &= \sum_i (D_i \log v)(m_\homg) \langle \dd m_\homg^{\xi_i}, \dd \vph\rangle_\homg \\
    &= \sum_i (D_i \log v)(m_\homg) (-\CL_{J \xi_i} \vph) \\
    &= \sum_i (D_i \log v)(m_\homg) (m_\homg^{\xi_i} - m_\om^{\xi_i}).
\end{align*}
\endgroup
Since $\im(m_\homg) = \Pol_{\af} = \im(m_\om)$ by our normalization, $m_\homg^{\xi_i}(X)$ and $m_\om^{\xi_i}(X)$ are uniformly bounded, and thus, 
\begin{equation}\label{eq:dlogv_dphi}
    |\langle \dd \log v(m_\homg), \dd \vph\rangle_\homg| \leq 2 C_v C_\Pol
\end{equation}
Putting $I_2$, \eqref{eq:P_est}, and \eqref{eq:dlogv_dphi} together, we obtain that  
\begingroup
\allowdisplaybreaks
\begin{equation}\label{eq:I2_est}
\begin{split}
    I_2 &\geq -\frac{b (n-1 + 2C_v C_\Pol)}{a} \int_X H(\tr_\homg \om) v(m_\homg) \homg^n \\
    &= -\frac{b C_3}{a} \int_X u^{2p+1} v(m_\homg) \homg^n, 
\end{split}
\end{equation}
where $C_3 = C_3(n, C_v, C_\Pol) > 0$ is a uniform constant. 
\endgroup

\medskip
\noindent{\bf Analyzing $I_1$.}
We now focus on estimating $I_1$. 
As $\dd \Om = 0$, by Stokes' theorem,
\begingroup
\allowdisplaybreaks
\begin{equation}\label{eq:mix2}
\begin{split}
    &\int_X H (\Dt_\Om G) v(m_\homg) \homg^n 
    = \int_X H e^{\log v(m_\homg)} \ddc G \w \Om \\
    &\qquad= - \int_X v(m_\homg) \dd H \w \dc G \w \Om 
    - \int_X H e^{\log v(m_\homg)} \dd \log v(m_\homg) \w \dc G \w \Om\\
    &\qquad= -\int_X \llangle \dd H, \dd G\rrangle_\Om v(m_\homg)\homg^n 
    - \int_X H \llangle \dd \log v(m_\homg), \dd G \rrangle_\Om v(m_\homg) \homg^n,
\end{split}
\end{equation}
\endgroup
where 
\[
    \llangle\dd f, \dd g\rrangle_\Om 
    := \frac{\dd f \w \dc g \w \Om}{\homg^n}
    \quad\text{and}\quad
    \|\dd f\|_\Om^2 := \llangle \dd f, \dd f\rrangle_\Om.
\]
Substituting \eqref{eq:mix2} into \eqref{eq:mix1}, we obtain that
\begin{align*}
    I_1 &= \frac{1}{a} \int_X \llangle \dd H, \dd G \rrangle_\Om v(m_\homg) \homg^n \\ 
    &\qquad + \frac{1}{a} \int_X H \left[\llangle \dd \log v(m_\homg), \dd G\rrangle_\Om - \tr_\homg \om \langle \dd \log v(m_\homg), \dd G\rangle_\homg\right] v(m_\homg) \homg^n.
\end{align*}
Recall that $H = e^G u^{2p}$. 
We have 
\begin{equation}\label{eq:dH_dG}
    \dd H = e^G u^{2p} \dd G + 2p e^G u^{2p-1} \dd u 
    = H \dd G + 2p H \dd \log u,
\end{equation}
and thus, 
\[
    \llangle \dd H, \dd G\rrangle_\Om 
    = H \llangle \dd G, \dd G\rrangle_\Om + 2p H \llangle \dd \log u, \dd G\rrangle_\Om.  
\]
We then get
\begin{equation}\label{eq:mix3}
\begin{split}
    I_1 
    &= \frac{1}{a} \int_X H \left[\| \dd G\|_\Om^2 + 2p \llangle \dd \log u, \dd G\rrangle_\Om\right] v(m_\homg) \homg^n\\
    &\qquad + \frac{1}{a} \int_X H \left[ \llangle \dd \log v(m_\homg), \dd G\rrangle_\Om - \tr_\homg \om \langle \dd \log v(m_\homg), \dd G\rangle_\homg \right]  v(m_\homg) \homg^n.
\end{split}
\end{equation}
By the definition of $\llangle \cdot, \cdot \rrangle_\Om$ and Lemma~\ref{lem:tensor_eq_Sze}, 
\begingroup
\allowdisplaybreaks
\begin{equation}\label{eq:llrr_trlr}
\begin{split}
    &\llangle \dd f, \dd g\rrangle_\Om - \tr_\homg \om \langle \dd f, \dd g\rangle_\homg \\
    &\qquad = n(n-1) \frac{\dd f \w \dc g \w \om \w \homg^{n-2}}{\homg^n} - (\tr_\homg \dd f \w \dc g) (\tr_\homg \om) \\
    &\qquad = - \langle \dd f \w \dc g, \om \rangle_\homg
\end{split}
\end{equation}
\endgroup
and one can also infer the following estimate from \eqref{eq:llrr_trlr}
\begin{equation}\label{eq:llrr_est}
    \| \dd f\|_\Om^2 
    \leq (\tr_\homg \om) |\dd f|_\homg^2
    \quad\text{and}\quad
    \langle\dd f \w \dc f, \om \rangle_\homg 
    \leq (\tr_\homg \om) |\dd f|_\homg^2. 
\end{equation} 
This turns \eqref{eq:mix3} to the following
\begin{equation}\label{eq:mix4}
\begin{split}
    I_1 
    &= \frac{1}{a} \int_X H \left[\|\dd G\|_\Om^2 + 2p \llangle \dd \log u, \dd G\rrangle_\Om \right] v(m_\homg) \homg^n\\
    &\qquad - \frac{1}{a} \int_X H \langle \dd \log v(m_\homg) \w \dc G, \om\rangle_\homg v(m_\homg) \homg^n.
\end{split}
\end{equation}

\smallskip
We shall first deal with the first term in \eqref{eq:mix4}. 
Note that 
\begin{align*}
    \| \dd G\|_\Om^2 + 2p \llangle \dd \log u, \dd G\rrangle_\Om
    &= \|\dd (G + p \log u)\|_\Om^2 - p^2 \|\dd \log u\|_\Om^2\\
    &\geq - p^2 (\tr_\homg \om) |\dd \log u|_\homg^2,
\end{align*}
where the second inequality follows from \eqref{eq:llrr_est}. 
Thus, as $H (\tr_\homg \om) = u^{2p+1}$, 
\begin{equation}\label{eq:mix5}
    I_1 
    \geq - \frac{p^2}{a} \int_X u^{2p+1} |\dd \log u|_\homg^2 v(m_\homg) \homg^n
    - \frac{1}{a} \int_X H \langle \dd \log v(m_\homg) \w \dc G, \om\rangle_\homg v(m_\homg) \homg^n. 
\end{equation}
We now concentrate on the second term in \eqref{eq:mix5}. 
By \eqref{eq:dH_dG}, 
\begingroup
\allowdisplaybreaks
\begin{equation}\label{eq:mix6}
\begin{split}
    I_1 
    &\geq - \frac{p^2}{a} \int_X u^{2p+1} |\dd \log u|_\homg^2 v(m_\homg) \homg^n \\
    &\qquad - \frac{1}{a} \int_X \langle \dd \log v(m_\homg) \w \dc H, \om\rangle_\homg v(m_\homg) \homg^n \\
    &\qquad + \frac{2p}{a} \int_X H \langle \dd\log v(m_\homg) \w \dc \log u, \om \rangle_\homg v(m_\homg) \homg^n.  
\end{split}
\end{equation}
\endgroup
As $\dd m_\om^{\xi_i} = -\om(\xi_i,\cdot) = -g_\om(J\xi_i,\cdot)$, we have $\nabla_\om m_\om^{\xi_i} = -J\xi_i$ and then a pointwise computation yields the following equalities 
\begin{equation}\label{eq:dmw_dm}
\begin{split}
    \langle\dd m_\homg^{\xi_i} \w \dc f, \om \rangle_\homg 
    &= g_\om(\nabla_\homg m_\homg^{\xi_i}, \nabla_\homg f) 
    = g_\om(-J\xi_i, \nabla_\homg f)\\ 
    &= g_\om(\nabla_\om m_\om^{\xi_i}, \nabla_\homg f)
    = \dd m_\om^{\xi_i} (\nabla_\homg f)
    = \langle \dd m_\om^{\xi_i}, \dd f\rangle_\homg
\end{split}
\end{equation}
and thus, 
\[
    \langle \dd m_\homg^{\xi_i} \w \dc H, \om \rangle_\homg 
    = \langle \dd m_\om^{\xi_i}, \dd H \rangle_\homg. 
\]
By the above equality and Stokes' theorem, we then obtain  
\begingroup
\allowdisplaybreaks
\begin{align*}
    &- \int_X \langle \dd \log v(m_\homg) \w \dc H, \om\rangle_\homg v(m_\homg) \homg^n\\ 
    &\qquad=- \sum_i \int_X (D_i \log v)(m_\homg) \langle \dd m_\homg^{\xi_i} \w \dc H, \om \rangle_\homg v(m_\homg) \homg^n\\
    &\qquad= - \sum_i \int_X (D_i \log v)(m_\homg) \langle \dd m_\om^{\xi_i}, \dd H\rangle_\homg v(m_\homg) \homg^n \\
    &\qquad= - n \sum_i \int_X (D_i \log v)(m_\homg) e^{\log v(m_\homg)} \dd H \w \dc m_\om^{\xi_i} \w \homg^{n-1} \\
    &\qquad= n \sum_i \int_X H (D_i \log v)(m_\homg) e^{\log v(m_\homg)} \ddc m_\om^{\xi_i} \w \homg^{n-1}\\ 
    &\qquad\qquad+ n \sum_{i,j} \int_X H \left[D^2_{ij} \log v + (D_i \log v) (D_j \log v)\right](m_\homg) e^{\log v(m_\homg)} \dd m_{\homg}^{\xi_i} \w \dc m_\om^{\xi_j} \w \homg^{n-1}. 
\end{align*}
\endgroup
Note that 
\[
    n \frac{\dd m_\homg^{\xi_i} \w \dc m_\om^{\xi_j} \w \homg^{n-1}}{\homg^n} 
    = g_\om(\xi_i,\xi_j)
    \quad\text{and}\quad
    \ddc m_\om^{\xi_i} = - \CL_{J\xi_i} \om.  
\]
Hence, 
\begingroup
\allowdisplaybreaks
\begin{equation}\label{eq:mix7}
\begin{split}
    &-\int_X \langle\dd \log v(m_\homg) \w \dc H, \om \rangle_\homg v(m_\homg) \homg^n \\
    &\qquad \geq - C_v A_1 \int_X H (\tr_\homg \om) v(m_\homg) \homg^n 
    - C_v^2 A_2 \int_X H v(m_\homg) \homg^n \\
    &\qquad \geq - \left(C_v A_1 + \frac{C_v^2 A_2}{c_0}\right) \int_X u^{2p+1} v(m_\homg) \homg^n.
\end{split}
\end{equation}
\endgroup
Combining \eqref{eq:mix6} and \eqref{eq:mix7}, we deduce that 
\begin{equation}\label{eq:mix8}
\begin{split}
    I_1 
    &\geq - \frac{p^2}{a} \int_X u^{2p+1} |\dd \log u|_\homg^2 v(m_\homg) \homg^n \\
    &\qquad - \frac{1}{a} C_4 \int_X u^{2p+1} v(m_\homg) \homg^n 
    + \frac{2p}{a} \int_X H \langle \dd\log v(m_\homg) \w \dc \log u, \om \rangle_\homg v(m_\homg) \homg^n,
\end{split}
\end{equation}
where $C_4 = C_4(A_1, A_2, C_v, c_0) > 0$ is a uniform constant. 
We remark that the pairing $\langle \dd f \w \dc g,\om \rangle_\homg$ is a positive definite non-degenerate bilinear pairing. 
Also, using Claim~\ref{claim:dmdm_ddcm}, 
\begingroup
\allowdisplaybreaks
\begin{equation}\label{eq:mix9}
\begin{split}
    &\langle \dd \log v(m_\homg) \w \dc \log v(m_\homg),\om \rangle_\homg \\
    &\qquad= \left|\sum_{i,j} [D_i \log v][D_j \log v](m_\homg) \langle \dd m_\homg^{\xi_i} \w \dc m_\homg^{\xi_j}, \om \rangle_\homg \right|\\
    &\qquad= \left|\sum_{i,j} [D_i \log v][D_j \log v](m_\homg) g_\om(\xi_i,\xi_j)\right|
    \leq C_v^2 A_2.
\end{split}
\end{equation}
\endgroup
Therefore, by Cauchy--Schwarz inequality, \eqref{eq:mix9}, and \eqref{eq:llrr_est}, we have 
\begingroup
\allowdisplaybreaks
\begin{align*}
    &2p |\langle\dd \log v(m_\homg) \w \dc \log u, \om\rangle_\homg| \\ 
    &\qquad \leq 2p \sqrt{\langle\dd \log v(m_\homg) \w \dc \log v(m_\homg), \om\rangle_\homg \langle\dd \log u \w \dc \log u, \om\rangle_\homg}\\
    &\qquad\leq 2p \sqrt{C_v^2 A_2 (\tr_\homg \om) |\dd \log u|_\homg^2} 
    \leq \tr_\homg \om |\dd \log u|_\homg^2 + C_v^2 A_2 p^2.
\end{align*}
\endgroup
The last term in \eqref{eq:mix8} enjoys the following estimate 
\begin{equation}\label{eq:mix10}
\begin{split}
    &2p \int_X H \langle\dd \log v(m_\homg) \w \dc \log u, \om \rangle_\homg v(m_\homg) \homg^n \\ 
    &\qquad\geq - \int_X u^{2p+1} |\dd \log u|_\homg^2 v(m_\homg) \homg^n - \frac{C_v^2 A_2}{c_0} p^2 \int_X u^{2p+1} v(m_\homg) \homg^n; 
\end{split}
\end{equation}
hence, 
\begin{equation}\label{eq:I1_est}
    I_1 \geq - \frac{p^2+1}{a} \int_X u^{2p+1} |\dd \log u|_\homg^2 v(m_\homg) \homg^n - C_{5,a,p} \int_X u^{2p+1} v(m_\homg) \homg^n,
\end{equation}
where $C_{5,a,p} > 0$ depends only on $C_4$, $a$, $p$, $C_v$, $A_2$ and $c_0$. 

\subsubsection{Conclusion} 
Putting \eqref{eq:IPP2}, \eqref{eq:mix1}, \eqref{eq:I2_est}, and \eqref{eq:I1_est} together, we end up with the following estimate 
\begin{align*}
    0 &\geq \left(2p - \frac{p^2+1}{a}\right) \int_X u^{2p+1} |\dd \log u|_\homg^2 v(m_\homg) \homg^n \\
    &\qquad+ (b - (a+1) A_0 - C_1) \int_X u^{2p+1} (\tr_\homg \om) v(m_\homg) \homg^n \\
    &\qquad- C_{6,a,b,p} \int_X u^{2p+1} v(m_\homg) \homg^n,
\end{align*}
where $C_{6,a,b,p}$ depends only on $a$, $b$, $C_3$, and $C_{5,a,p}$. 
Choosing $a = \frac{p^2+1}{2p}$ and $b \geq (a + 1)A_0 + C_1 + 1$ and applying H\"older inequality, we get that 
\begin{align*}
    &\int_X (\tr_\homg \om)^{2p+2} v(m_\homg) \homg^n 
    \leq C_p \int_X (\tr_\homg \om)^{2p+1} v(m_\homg) \homg^n \\
    &\qquad\leq C_p \left(\int_X (\tr_\homg \om)^{2p+2} v(m_\homg) \homg^n\right)^{\frac{2p+1}{2p+2}} \left(\int_X v(m_\homg) \homg^n\right)^{\frac{1}{2p+2}};
\end{align*}  
here $C_p>0$ depends only on $p$, $C_{6,a,b,p}$ and $C_0$. 
Hence, we end up with 
\[
    \int_X (\tr_\homg \om)^{2p+2} e^F \om^n 
    \leq C_p^{2p+2} \int_X e^F \om^n 
\]
and this gives the first estimate in Theorem~\ref{thm:main1}. 
The second estimate in Theorem~\ref{thm:main1} simply comes from the fact that 
\begin{equation}\label{eq:det_tr_ineq}
    \tr_\om \homg \leq \frac{\homg^n}{\om^n} (\tr_\homg \om)^{n-1}. 
\end{equation}
Therefore, we obtain 
\begin{align*}
    \int_X (\tr_\om \homg)^{\frac{2p+2}{n-1}} \om^n
    &\leq \int_X \left(\frac{e^F}{v(m_\homg)}\right)^{\frac{2p+2}{n-1}} (\tr_\homg \om)^{2p+2} \om^n \\
    &\leq \left(\frac{e^{C_0}}{\min_{\Pol_\af}v}\right)^{\frac{2p+2}{n-1}} \int_X (\tr_\homg\om)^{2p+2} \om^n.
\end{align*}
This concludes the proof of Theorem~\ref{thm:main1}.

\subsection{Integral estimates for the Laplacian along the continuity path}
To show that the coercivity of the weighted Mabuchi functional implies the existence of weighted cscK metrics, Theorem~\ref{thm:main1} is not sufficient. 
As in \cite{DiNezza_Jubert_Lahdili_2024_b}, one requires to consider the following twisted weighted cscK equation: 
\begin{equation}\tag{twcscK}\label{eq:twcscK}
\begin{cases}
    v(m_{\om_\vph}) \om_\vph^n = e^F \om^n, \\
    \Dt_{\om_\vph,v} (F + f) = -w(m_{\om_\vph}) + \tr_{\om_\vph,v} (\Ric(\om) - \beta),
\end{cases}
\end{equation}
where $\bt$ is a $\BT$-invariant closed $(1,1)$-form and $f$ is a $\T$-invariant smooth function such that $\bt + \ddc f \geq 0$. 
Note that here we fix a moment map $m_\bt$ for $\bt$ and consider the normalization $m_{\beta + \ddc f} = m_\beta + \dc f$. 
As before, we denote $\homg = \om_\vph$.

Twisted weighted cscK equations arise naturally in the continuity method of \cite{DiNezza_Jubert_Lahdili_2024_b}. More precisely, the presence of the complexified torus $\T^{\mathbb C}$ in the coercivity condition requires one to work modulo its action along the continuity path, which  introduces an additional twisting term in the weighted cscK equation.

\smallskip
The $\log$-concavity of $v$ was used only in the proof of \cite[Theorem~2.6]{DiNezza_Jubert_Lahdili_2024_b} to obtain integral estimates for the Laplacian of $\vph$. 
Hence, we concentrate mainly on proving this part without assuming $v$ is $\log$-concave.

\begin{thm}\label{thm:main2}
Let $(\vph, F)$ be a pair solving \eqref{eq:twcscK}. 
For any $k \geq 1$, there exists a constant $C_k > 0$ such that 
\[
    \int_X e^{(k-1)f} (\tr_\om \homg)^k \om^n \leq C_k, 
\]
where $C_k$ depends only on $k$, $\om$, $\beta$, $\|\vph\|_{L^\infty}$, $\|F+f\|_{L^\infty}$, 
$\displaystyle \max_{\Pol_\af} \left(|v| + \sum_i |D_i \log v| + \sum_{i,j} |D_{ij}^2 \log v|\right)$, 
$\displaystyle \min_{\Pol_\af} v$, 
$\displaystyle \max_{\Pol_\af} |w|$, 
$\displaystyle \max_i \max_{\eta \in \Pol_\af}|\langle\xi_i,\eta\rangle_{\mathfrak{t},\mathfrak{t}^\vee}|$, and 
$\displaystyle \max_i \max_{x \in X}|m_{\Ric(\om)}^{\xi_i}(x)|$. 
Moreover, if $\int_X e^{-p_0 f} \om^n \leq M$ for some $p_0 > 1$, then for any $q \in (1,p_0)$, 
\[
    \|\tr_\om \homg\|_{L^q(\om^n)} \leq C_q, 
\]
where $C_q > 0$ depends only on $q$, $p_0$, $M$, and the constants appearing above. 
\end{thm}

\begin{rmk}
This result removes the $\log$-concavity assumption in \cite[Theorem~2.6 and Corollary~2.7]{DiNezza_Jubert_Lahdili_2024_b}. 
We remark that the integral estimate for the gradient of $F+f$ in \cite[Theorem~2.1]{DiNezza_Jubert_Lahdili_2024_b} follows from \cite[Theorem~2.6]{DiNezza_Jubert_Lahdili_2024_b}, whose proof is the only step in which $\log$-concavity enters the game. 
As Theorem~\ref{thm:main2} removes the $\log$-concavity in \cite[Theorem~2.6]{DiNezza_Jubert_Lahdili_2024_b}, the $\log$-concavity assumption in the statement of \cite[Theorem~2.1]{DiNezza_Jubert_Lahdili_2024_b} is not needed. 
\end{rmk}

\begin{proof}[Proof of Theorem~\ref{thm:main2}]
Since most computations are similar to those in Section~\ref{sec:pf_main1} and Section~\ref{sec:auxiliary}, we highlight the main differences. 

\smallskip
Here, instead of assuming $\|F\|_{L^\infty} \leq C_0$, we assume that 
\begin{equation}\label{eq:conti_C0}
    \|\vph\|_{L^\infty} + \|F+f\|_{L^\infty} \leq C_0
\end{equation}
for some uniform constant $C_0 > 0$; also, $\sup_X \vph = \sup_X f = 0$. 
From \eqref{eq:twcscK}, we have 
\[
    \Ric(\homg) = \Ric(\om) - \ddc (F+f) + \ddc f + \ddc \log v(m_\homg).
\]
Following estimates similar to those in \eqref{eq:ess1} and \eqref{eq:ess2}, one can derive 
\begin{equation}\label{eq:conti_wLap_logtr}
\begin{split}
    \Dt_{\homg,v} \log \tr_\homg \om 
    &\geq - \frac{\langle \ddc (F+f), \om \rangle_\homg}{\tr_\homg \om} + \frac{\langle \ddc f, \om \rangle_\homg}{\tr_\homg \om} \\ 
    &\qquad - (A_0 + B) \tr_\homg \om - C_v A_1 - C_v A_2 \frac{1}{\tr_\homg \om}
\end{split}
\end{equation}
We remark that as before $\tr_\homg \om \geq n \sqrt[n]{\frac{v(m_\homg)}{e^F}}$. 
However, the function $F$ has no uniform control here (only $(F+f)$ is under control), so we shall keep $\frac{1}{\tr_\homg \om}$ throughout the estimates.

\smallskip
Let $C_\bt > 0$ be a constant so that $-C_\beta \om \leq \beta \leq C_\beta \om$. 
As $\beta + \ddc f \geq 0$, one gets $\ddc f \geq -C_\beta \om$. 
Hence, from \eqref{eq:conti_wLap_logtr}, we have
\begin{equation*}
    \Dt_{\homg,v} \log \tr_\homg \om 
    \geq - \frac{\langle \ddc (F+f), \om \rangle_\homg}{\tr_\homg \om} - (A_0 + B + C_\beta) \tr_\homg \om - C_v A_1 - C_v A_2 \frac{1}{\tr_\homg \om},
\end{equation*}

\smallskip
We consider 
\[
    G = -(a (F+f) + b \vph),\quad
    u = e^G \tr_\homg \om, \quad\text{and}\quad
    H = e^G u^{2p}. 
\]
The constants appearing below all depend on some of the constants appearing at the beginning of Section~\ref{sec:pf_main1}, and possibly also $a$ and $b$; in that case, we shall include $a$ and $b$ in the subscripts.  

\smallskip
Following computations similar to those used to obtain \eqref{eq:wLap_logu} and \eqref{eq:IPP2}, one can infer that 
\begin{align*}
    \Dt_{\homg,v} \log u 
    &\geq \left[b - (a + 1) (A_0 + C_\bt) - C_1\right] \tr_\homg \om - C_{2,a,b} - \frac{C_3}{\tr_\homg \om}\\
    &\qquad+ \frac{\Dt_\Om(F+f)}{ \tr_\homg \om} + \langle \dd \log v(m_\homg), \dd (F+f)\rangle_\homg,
\end{align*} 
and  
\begin{equation}\label{eq:conti_IPP}
\begin{split}
    0 &\geq 2p \int_X u^{2p+1} |\dd \log u|_\homg^2 v(m_\homg) \homg^n \\
    &\qquad + \left[b - (a+1) (A_0 + C_\beta) -C_1\right] \int_X u^{2p+1} (\tr_\homg \om) v(m_\homg) \homg^n\\
    &\qquad - C_{2,a,b} \int_X u^{2p+1} v(m_\homg) \homg^n 
    - C_3 \int_X \frac{u^{2p+1}}{\tr_\homg \om} v(m_\homg) \homg^n\\
    &\qquad + \int_X H [\Dt_\Om(F+f) + (\tr_\homg \om) \langle \dd \log v(m_\homg), \dd (F + f)\rangle_\homg] v(m_\homg) \homg^n.
\end{split}
\end{equation}
Then we follow the argument in Section~\ref{eq:analyze_I}. 
Note that in Section~\ref{eq:analyze_I}, the only two parts where $\tr_\homg \om \geq c_0$ was used are \eqref{eq:mix7} and \eqref{eq:mix10}. 
Hence, when one replaces $F$ by $(F+f)$ and goes through almost the same argument in Section~\ref{eq:analyze_I}, one only needs to keep $\frac{1}{\tr_\homg \om}$ in the integration of these two estimates. 
This allows us to obtain the following 
\begin{equation}
\begin{split}
    0 &\geq \left(2p - \frac{p^2+1}{a}\right) \int_X u^{2p+1} |\dd \log u|_\homg^2 v(m_\homg) \homg^n\\
    &\qquad + \left[b - (a+1) (A_0+C_\beta) - C_1\right] \int_X u^{2p+1} (\tr_\homg \om) v(m_\homg) \homg^n \\ 
    &\qquad - C_{4,a,b,p} \int_X u^{2p+1} v(m_\homg) \homg^n
    - C_{3,a,p} \underbrace{\int_X \frac{u^{2p+1}}{\tr_\homg \om} v(m_\homg) \homg^n}_{= \int_X u^{2p}e^G v(m_\homg) \homg^n}.
\end{split}
\end{equation} 
Choosing $a = \frac{p^2+1}{2p}$ and $b \geq (a+1) (A_0 + C_\beta) + C_1 + 1$, we obtain that 
\[
    \int_X (\tr_\homg \om)^{2p+2} v(m_\homg) \homg^n 
    \leq C_{5,p} \int_X (\tr_\homg \om)^{2p+1} v(m_\homg) \homg^n + C_{6,p} \int_X (\tr_\homg \om)^{2p} v(m_\homg) \homg^n,
\]
and after applying H\"older inequality, one can derive that 
\[
    \int_X (\tr_\homg \om)^{2p+2} v(m_\homg) \homg^n \leq C_{7,p}.
\]
Using \eqref{eq:det_tr_ineq}, \eqref{eq:conti_C0} and \eqref{eq:twcscK}, one can derive that, for $k \geq 1$,  
\begin{align*}
    \int_X (\tr_\om \homg)^k e^{(k-1)f} \om^n
    &\leq \int_X \frac{e^{(k-1)F}}{v^k(m_\homg)} e^{(k-1)f} (\tr_\homg \om)^{k(n-1)} e^F \om^n\\ 
    &\leq \frac{1}{\min_{\Pol_\af} v^k} \int_X (\tr_\homg \om)^{k(n-1)} e^{(k-1)(F+f)} v(m_\homg) \homg^n \leq C_k.
\end{align*}
This shows the first part of Theorem~\ref{thm:main2}.

\smallskip
The last part follows similarly from the proof of \cite[Corollary~2.7]{DiNezza_Jubert_Lahdili_2024_b}. 
Assume that $\int_X e^{-p_0 f} \om^n < +\infty$ for some $p_0 > 1$. 
For $1 < q < p_0$, choose $k = \frac{q(p_0-1)}{p_0-q} > q$.
Then, applying H\"older inequality, we obtain that 
\begin{align*}
    \int_X (\tr_\om \homg)^q \om^n 
    \leq \left(\int_X e^{(k-1)f} (\tr_\om \homg)^k \om^n\right)^{q/k} \left(\int_X e^{-\frac{q(k-1)}{k-q}f} \om^n\right)^{1-q/k}  
\end{align*}
and by our choice $\frac{q(k-1)}{k-q} = p_0$.
This concludes the proof of Theorem~\ref{thm:main2}. 
\end{proof}

Theorem~\ref{thm:main2} helps to get rid of the $\log$-concavity assumption in \cite{DiNezza_Jubert_Lahdili_2024, Han_Liu_2024, DiNezza_Jubert_Lahdili_2024_b}. 
We thus generalize \cite[Theorems~1.1 and 1.2]{DiNezza_Jubert_Lahdili_2024_b} to the following results: 

\begin{thm} 
Fix a K\"ahler class $\af$ on $X$. 
Let $\BT \subset \Aut_{\red}(X)$ be a real torus in the reduced group of automorphisms of $X$ and let $(v,w) \in \CC_{>0}^\infty(\Pol_\af) \times \CC^\infty(\Pol_\af)$ be two weight functions satisfying \eqref{norma:weigh}.
If the weighted Mabuchi functional $\M_{v,w}$ is relatively $\BT^\BC$-coercive, then there exists a $(v,w)$-cscK metric in $\af$. 
\end{thm}

Combining this result with \cite[Theorem~1]{Apostolov_Jubert_Lahdili_2023} proves Theorem~\ref{t:A}.

\section{Laplacian estimate for weighted K\"ahler--Ricci solitons}\label{sec:wKRS}
The estimates in the previous section of course cover the weighted K\"ahler--Ricci solitons. 
However, to obtain a uniform Laplacian estimate, one also needs to go through \cite{DiNezza_Jubert_Lahdili_2024, Han_Liu_2024}, which makes the process more involved. 
As the weighted K\"ahler--Ricci soliton equation is of second order, one may wonder if there is a more direct way to obtain the Laplacian estimate without assuming $v$ to be $\log$-concave. 
We therefore decide to present a much simpler argument for the weighted K\"ahler--Ricci soliton in this section. 

\smallskip
Let $X$ be an $n$-dimensional Fano manifold. 
Fix a $\BT$-invariant K\"ahler metric $\om$ in $4\pi c_1(X)$. 
Recall that the $v$-soliton equation is 
\[
    \Ric(\om_\vph) = \om_\vph + \ddc \log v(m_{\om_\vph})
\]
and it can be reformulated into the following complex Monge--Amp\`ere type equation
\begin{equation}\tag{wKRS}\label{eq:wKRS}
    v(m_{\om_\vph}) \om_\vph^n = e^{-\vph + h_\om} \om^n, 
\end{equation}
where $h_\om$ is a Ricci potential of $\om$ such that $\Ric(\om) - \om = \ddc h_\om$. 
Set $\homg = \om_\vph$ and $\af = 4 \pi c_1(X)$. 

\begin{thm}\label{thm:main3}
Let $\vph$ be a solution to \eqref{eq:wKRS}. 
Then there is a uniform constant $C > 0$ such that 
\[
    \tr_\om \homg \leq C, 
\]
where $C$ depends only on $\om$, $\|\vph\|_{L^\infty}$, $\|h_\om\|_{L^\infty}$, 
$\displaystyle \max_{\Pol_{\af}} \left\{|v| + |D_i \log v| + |D^2_{ij} \log v|\right\}$,  
$\displaystyle \min_{\Pol_{\af}} v$, and 
$\displaystyle \max_i \max_{\eta \in \Pol_{\af}}|\langle\xi_i,\eta\rangle_{\mathfrak{t},\mathfrak{t}^\vee}|$. 
\end{thm}

\begin{proof}
Let $B > 0$ be a constant such that $\Bisec(\om) \leq B$.  
From \eqref{eq:wKRS}, we have 
\[
    \Ric(\homg) = \homg + \ddc \log v(m_\homg). 
\] 
Applying Chern--Lu inequality (Lemma~\ref{lem:Chern-Lu}), we obtain 
\begin{equation}\label{eq:Lap_1}
\begin{split}
    \Dt_\homg \log \tr_\homg \om
    &\geq \frac{\langle \Ric(\homg), \om \rangle_\homg}{\tr_\homg \om} - B \tr_\homg \om\\
    &= \frac{\langle \homg + \ddc \log v(m_\homg), \om \rangle_\homg}{\tr_\homg \om} - B \tr_\homg \om\\
    &= \frac{\langle \ddc \log v(m_\homg), \om \rangle_\homg}{\tr_\homg \om} - B \tr_\homg \om + 1.
\end{split}
\end{equation}
By Claim~\ref{claim:dmdm_ddcm}, we have 
\[
    \langle \dd m_\homg^{\xi_i} \w \dc m_\homg^{\xi_j} , \om\rangle_\homg = g_\om (\xi_i,\xi_j)
    \quad\text{and}\quad 
    \langle \ddc m_\homg^{\xi_i}, \om \rangle_\homg 
    = \CL_{J\xi_i} \tr_\homg \om - \tr_\homg(\CL_{J\xi_i} \om),
\]
and thus, 
\begin{align*}
    \langle \ddc \log v(m_\homg), \om \rangle_\homg 
    &= \sum_{i,j} (D_{ij}^2 \log v)(m_\homg) g_\om(\xi_i,\xi_j)\\
    &\qquad + \sum_i (D_i \log v)(m_\homg) (\CL_{J\xi_i} \tr_\homg \om - \tr_\homg(\CL_{J\xi_i} \om))\\
    &\geq \sum_i (D_i \log v)(m_\homg) \CL_{J\xi_i} \tr_\homg \om - C_v A_1 \tr_\homg \om - C_v A_2.
\end{align*}
Substituting the above inequality into \eqref{eq:Lap_1}, we get
\begin{align*}
    \Dt_\homg \log \tr_\homg \om 
    &\geq \frac{\sum_i (D_i \log v)(m_\homg) (\CL_{J\xi_i} \tr_\homg\om)}{\tr_\homg \om} - B \tr_\homg \om - C_1\\
    &= \sum_i (D_i \log v)(m_\homg) (\CL_{J\xi_i} (\log \tr_\homg \om)) - B \tr_\homg \om - C_1
\end{align*}
and thus, 
\begin{align*}
    \Dt_\homg (\log \tr_\homg \om - A \vph)
    &\geq \sum_i (D_i \log v)(m_\homg) (\CL_{J\xi_i} (\log \tr_\homg \om)) - B \tr_\homg \om - C_1 - A(n-\tr_\homg \om)\\
    &= \sum_i (D_i \log v)(m_\homg) (\CL_{J\xi_i} (\log \tr_\homg \om)) + (A-B) \tr_\homg \om - (C_1 + An). 
\end{align*}
Set $H = \log \tr_\homg \om - A \vph$ and $A = B+1$. 
At a maximum point $x_0$ of $H$, we have 
\[
    (\Dt_\homg H)(x_0) \leq 0
\]
and 
\begin{align*}
    &\CL_{J\xi_i} H (x_0) = 0 \\
    \iff & \CL_{J\xi_i} (\log \tr_\homg \om)(x_0) = A \CL_{J\xi_i} \vph (x_0)
    = -A \iota_{\xi_i} \dc \vph (x_0)
    = -A (m_\homg^{\xi_i} - m_\om^{\xi_i})(x_0). 
\end{align*}
Since $\im(m_\homg) = \Pol_{\af} = \im(m_\om)$ by our normalization, $m_\homg^{\xi_i}(X)$ and $m_\om^{\xi_i}(X)$ are uniformly bounded. 
We have the following at $x_0$, 
\begin{align*}
    0 &\geq (\Dt_\homg H)(x_0)\\
    &\geq -A \sum_i (D_i \log v) (m_\homg^{\xi_i} - m_\om^{\xi_i})(x_0) + \tr_\homg \om(x_0) - (C_1 + An)\\
    &\geq \tr_\homg \om(x_0) - C_2.
\end{align*}
Then
\[
    \log \tr_\homg \om \leq H(x_0) + A\vph 
    \leq \log C_2 - A \vph(x_0) + A\vph
    \leq \log C_2 + 2 A \|\vph\|_{L^\infty}. 
\]
This provides a uniform Laplacian estimate of $\vph$ as follows
\[
    \tr_\om \homg \leq \frac{\homg^n}{\om^n} (\tr_\homg \om)^{n-1}
    \leq C_2 \frac{e^{-\vph + h_\om}}{\min_{\Pol_{\af}} v} e^{(n-1)2A\|\vph\|_{L^\infty}}
    \leq C_3 e^{[(n-1)2A+1]\|\vph\|_{L^\infty}}, 
\]
and this completes the proof. 
\end{proof}

Combining this with the regularizing argument in \cite[Proposition~3.8]{Han_Li_2023}, we obtain the following regularity result. 

\begin{prop}
Let $X$ be an $n$-dimensional Fano manifold, let $\BT$ be a maximal real torus in $\Aut_{\red}(X)$, and let $v \in \CC^\infty_{>0}(\Pol_{\af})$.  
If $\vph \in \CE^1(X,\om)^\BT$ is a finite energy solution to \eqref{eq:wKRS}, then $\vph \in \CH(X,\om)^\BT$. 
\end{prop}

\begin{rmk}
We remark that although \cite[Proposition~3.8]{Han_Li_2023} is stated without assuming that $v$ is $\log$-concave, the proof in \cite[(3.15)]{Han_Li_2023} appears to require this assumption. 
We emphasize that our argument also applies to $\log$ Fano pairs without any $\log$-concavity assumption on $v$. 
Since the proof is a direct combination of the strategy in the proof of Theorem~\ref{thm:main3} and the approximation argument in \cite[Proposition~3.8]{Han_Li_2023}, we leave the proof to interested readers. 
This therefore establishes a Yau--Tian--Donaldson type correspondence for general $v$-solitons on $\log$ Fano pairs. 
\end{rmk}

\section{Geometric applications of the weighted YTD correspondence}{\label{s:ytd}}

In this section, we apply the weighted Yau--Tian--Donaldson correspondence to characterize the existence of extremal Sasaki structures and conformally K\"ahler, Einstein--Maxwell metrics. We first recall the notion of relative uniform weighted K-stability. We then discuss the relation between weighted K-stability for suitable weight functions and the Collins--Sz\'ekelyhidi stability of polarized affine cones \cite{Collins_Szekelyhidi_2019}, following \cite{Apostolov_Calderbank_Legendre_2021}.

\subsection{Weighted stability of polarized varieties}
We begin by recalling relative uniform weighted K-stability for a smooth polarized variety.

\smallskip
Let $(X,\omega,\T)$ be an $n$-dimensional compact K\"ahler manifold. 
Suppose moreover that the K\"ahler class $\af = [\om]$ is polarized, i.e. $\alpha = 2 \pi c_1(L)$ is the first Chern class of an ample line bundle $L\rightarrow X$. 

\subsubsection{Test configurations}
We review the usual definition of a smooth equivariant K\"ahler test configuration.

\begin{defn}\label{defn:TC}
A $\T$-equivariant K\"ahler test configuration for a K\"ahler manifold $(X,\af,\T)$ consists of a compact K\"ahler manifold $\CX$ equipped with the following data: 
\begin{itemize}
    \item A surjective holomorphic map $\pi: \CX \to \BP^1$;
    \item A $\BC^*$-action $\rho$ covering the standard $\BC^*$-action on $\BP^1$, together with a $\T$-action commuting with $\rho$ and preserving the fibres of $\pi$;
    \item A $\BT \times \BC^*$-equivariant biholomorphism
    \[
        \Pi_0: \CX \setminus \pi^{-1}(0)
        \cong
        X \times (\BP^1\setminus\{0\});
    \]
    \item A K\"ahler class $\CA \in H^{1,1}(\CX,\BR)$ satisfying 
    \[
        (\Pi_0^{-1})^*(\CA)|_{X \times \{\tau\}} = \af,
    \]
    for every $\tau \in \BP^1 \setminus \{0\}$.
\end{itemize}
\end{defn}

A smooth polarized test configuration for $(X,L)$ is a smooth polarized variety $(\CX, \CL)$ such that, for some $r \in \BN^*$, $\left(\CX, \frac{2\pi}{r}c_1(\CL)\right)$ defines a K\"ahler test configuration for $(X, 2\pi c_1(L))$, and under the trivialization over $\BC^*$, $\CL|_{X \times \{\tau\}} \simeq L^r$ for every $\tau \in \BC^*$.

\subsubsection{Non-Archimedean $\J$-functional}
Without loss of generality, we may work with \emph{dominating} test configurations, meaning that $\Pi_0$ extends to a $\BT \times \BC^*$-equivariant morphism $\Pi: \CX \to X \times \BP^1$. 
For such a test configuration, the non-Archimedean $\J$-functional is defined by
\[
    \J^{\mathrm{NA}}(\CX,\CA)
    :=
    \frac{1}{\af^n} \left( \int_{\CX}(\Pi^*\af)^n \wedge \CA
    - \frac{1}{n+1} \int_{\CX} \CA^{n+1} \right).
\]

\smallskip
We next recall the construction of the $\xi$-twist $(\CX_\xi, \CA_\xi)$ of a smooth test configuration $(\CX, \CA)$ (cf. \cite{Hisamoto_16, Chi_Li_2022}). 
Let $\xi$ be an integral lattice element in the Lie algebra $\mathfrak{t}$ of $\T$. 
It induces a $\BC^*$-action $\rho_\xi$ on $\CX$. 
Composing the original $\BC^*$-action $\rho$ with $\rho_\xi$ gives a new $\BC^*$-action $\rho \cdot \rho_\xi$ on $\CX \setminus \pi^{-1}(\infty)$.

\smallskip
The twisted test configuration $(\CX_\xi, \CA_\xi)$ is obtained by taking the trivial compactification of $\CX \setminus \pi^{-1}(\infty)$ at infinity with respect to this new action (see \cite[Definition~2.7]{Boucksom_Hisamoto_Jonsson_2019}). 
We denote by $\BS^1_\xi$ the $\BS^1$-action on $\CX_\xi$ obtained by composing the original $\BS^1$-action with that induced by $\xi$. 
Thus, $(\CX_\xi, \CA_\xi)$ is a $\T \times \BS^1_\xi$-equivariant test configuration.

\smallskip
The construction extends similarly to rational elements $\xi \in \mathfrak{t}$. 
For irrational $\xi \in \mathfrak{t}$, the pair $(\CX_\xi,\CA_\xi)$ need not define a test configuration. 
Nevertheless, the quantity $\J^{\mathrm{NA}}(\CX_\xi,\CA_\xi)$ is well defined by continuity from rational elements $\xi \in \mathfrak{t}$.

\smallskip
Following \cite{Hisamoto_16, Chi_Li_2022}, for any dominating test configuration $(\CX,\CA)$, we define the reduced non-Archimedean $\J$-functional by
\begin{equation}\label{non:arc:j}
    \J^{\mathrm{NA}}_{\T^{\BC}} (\CX,\CA)
    :=
    \inf_{\xi\in\mathfrak{t}}
    \J^{\mathrm{NA}} (\CX_{\xi},\CA_{\xi}).
\end{equation}
This is the reduced functional used in the stability notion introduced below. 

\subsubsection{Weighted Donaldson--Futaki invariant}
Let $v>0$ and $w$ be smooth weight functions on the moment polytope $\Pol_\af$. 
For a smooth test configuration $(\CX,\CA)$, the $(v,w)$-weighted Donaldson--Futaki invariant introduced in \cite{Lahdili_2019} is 
\[
    \DF_{v,w}(\CX,\CA)
    =
    -\int_{\CX}
    \left(\Scal_v(\Om_{\CX}) - w(m_{\Om_{\CX}})\right) \Om_{\CX}^{[n+1]}
    + (8\pi) \Vol(X,\af),
\]
where $\Om_{\CX}$ is any $\T$-invariant K\"ahler metric in $\CA$ with $\Pol_\af$-normalized moment map $m_{\Om_{\CX}}$. 

\smallskip
When $w>0$, we define the $\T$-relative $(v,w)$-weighted Donaldson--Futaki invariant by
\begin{equation*}
    \DF^{\rel}_{v,w} := \DF_{v,w\ell_{v,w}^{\ext}},
\end{equation*}
where $\ell_{v,w}^{\mathrm{ext}}$ is the affine-linear function introduced in \eqref{lext}.

\begin{defn}{\label{d:tc}}
A polarized compact K\"ahler manifold $(X,\alpha,\T)$ is \emph{$\T$-relatively uniformly $(v,w)$-weighted K-stable} if there exists a constant $\lambda > 0$ such that
\[
    \DF^{\rel}_{v,w}(\CX,\CA)
    \geq \lambda \, \J^{\mathrm{NA}}_{{\T}^{\BC}} (\CX,\CA)
\]
for every $\T$-equivariant dominating smooth test configuration $(\CX, \CA)$. 
\end{defn}

\subsection{Application to Sasaki manifolds}{\label{s:sasaki}}
Let $(N, \CD, J_N)$ be a smooth compact strictly pseudoconvex CR manifold of real dimension $2n+1$. 
More precisely, $\CD \subset TN$ is a distribution of rank $2n$ and $J_N^2 = -\Id_\CD$. 
We further assume that $J_N$ is integrable on $\CD$, i.e. the complex distribution $\CD^{1,0} \subset TN \otimes \BC$ defined as the $+i$-eigenspace of $J_N$, is stable under the Lie bracket. 
Choose a contact form $\eta_N$ such that $\CD=\ker(\eta_N)$ and with sign chosen so that
\[
    (\dd \eta_N)_{|_{\CD}}(\cdot,J_N\cdot) > 0.  
\]
Suppose that there is a compact torus $\widehat{\T} \subset \Aut_{\mathrm{CR}}(N,\CD,J_N)$ acting on $N$, where $\Aut_{\mathrm{CR}}(N,\CD,J_N)$ is the group of diffeomorphisms preserving both $\CD$ and $J_N$.
We denote by $\widehat{\mathfrak{t}}$ the Lie algebra of $\widehat{\T}$. 

\subsubsection{Sasaki structures}
The contact Lie algebra is $\mathfrak{con}(N,\CD) := \Gamma(TN/\CD)$, equipped with the Jacobi bracket $[\xi,\chi] := \eta_{\CD}([\xi^\CD,\chi^\CD])$. 
Here $\eta_\CD: TN \to TN / \CD$ is the canonical projection, and $\xi^\CD, \chi^\CD$ are the unique lifts of $\xi$ and $\chi$ to vector fields on $N$ preserving $\CD$ (cf. \cite{Boyer_Galicki_2008} and
\cite[Lemma~1]{Apostolov_Calderbank_2021}). 
The Lie subalgebra 
\[
    \mathfrak{cr}(N,\CD,J_N) :=
    \left\{
    \xi \in \mathfrak{con}(N,\CD)
    \mid
    \mathcal L_{\xi^\CD}J_N=0
    \right\}
\]
corresponds to the CR vector fields. 
For a nowhere vanishing section $\xi$ of $TN/\CD$, let $\eta_\xi := \xi^{-1} \eta_\CD$, where $\xi^{-1} \in \Gm((TN/\CD)^\ast)$ is characterized by $\xi^{-1}(\xi) = 1$. 
Let $\mathfrak{con}_+(N,\CD)$ denote the open cone of positive sections, namely those satisfying ${\dd \eta_\xi}_{|\CD}(\cdot,J_N \cdot) > 0$.
The \emph{Sasaki cone} of $(N,\CD,J_N)$ is defined by
\[
    \mathfrak{cr}_+(N,\CD,J_N)
    :=
    \mathfrak{cr}(N,\CD,J_N)
    \cap
    \mathfrak{con}_+(N,\CD).
\]
An element $\xi \in \mathfrak{cr}_+(N,\CD,J_N)$ is called a \textit{Sasaki structure}, and $(N,\CD,J_N,\xi)$ is called a \textit{Sasaki manifold}. 

\subsubsection{Transverse K\"ahler structure}
The transverse K\"ahler form associated with $\xi$ is given by
\begin{equation}{\label{tran:k}}
    \om_\xi := {\dd \eta_\xi}_{|_{\CD}}.
\end{equation}
We denote by $\Scal(\om_\xi)$ the \textit{Tanaka--Webster scalar curvature} of the Sasaki manifold $(N,\CD,J_N,\xi)$.

\begin{defn}
A Sasaki manifold $(N,\CD,J_N,\xi)$ is said to be extremal if
\[
    \Scal(\om_\xi) \xi \in \mathfrak{cr}(N,\CD,J_N).
\]
\end{defn}

By \cite[Lemma~2]{Apostolov_Calderbank_2021}, the above definition is equivalent to requiring that $\Scal(\om_\xi)$ be a \textit{transverse Killing potential} for the transverse K\"ahler structure $(J_N,\omega_\xi)$.

\smallskip
We now fix a regular Sasaki structure $\chi \in \widehat{\mathfrak t}\cap \mathfrak{cr}_+(N,\CD,J_N)$, which means that $\chi^\CD$ generates a free $\BS^1$-action on $N$ denoted by $\BS^1_\chi$.
Define 
\begin{align*}
    &\mathfrak{con}^{\chi}(N,\CD)
    := \left\{ \xi \in \mathfrak{con}(N,\CD) \mid [\chi,\xi] = 0\right\}, \\ 
    &\mathfrak{cr}^{\chi}(N,\CD,J_N)
    := \mathfrak{con}^{\chi}(N,\CD) \cap \mathfrak{cr}(N,\CD,J_N), \\
    &\mathfrak{cr}^{\chi}_+(N,\CD,J_N)
    := \mathfrak{con}^{\chi}(N,\CD) \cap \mathfrak{cr}_+(N,\CD,J_N).
\end{align*}
Since $\chi$ is regular, there is a principal $\BS^1$-bundle structure 
\begin{equation*}
    \pi_\chi: N \to M_\chi := N/\BS^1_\chi
\end{equation*}
over the real $2n$-dimensional smooth manifold $M_\chi$. 
The transverse K\"ahler structure of $(N,\CD,J_N,\chi)$ induces a K\"ahler structure $(J_\chi,\om_\chi)$ on $M_\chi$ characterized by
\[
    \pi_\chi^* \omega_\chi = \dd \eta_\chi,
    \qquad
    \widetilde{J_\chi}=J_N,
\]
where $\widetilde{J_\chi}$ is the horizontal lift of $J_\chi$ to
$\CD$ with respect to the connection $\eta_\chi$. 
We denote the associated complex manifold by
\begin{equation*}
    X_\chi := (M_\chi, J_\chi). 
\end{equation*} 
By abuse of notation, we use the same symbol $\omega_\chi$ for the K\"ahler form on $X_\chi$ and for the transverse K\"ahler
form induced by $\chi$ as in \eqref{tran:k}.

\smallskip
Moreover, since $\widehat{\T}$ is abelian and contains $\BS^1_\chi$, its action descends to an action of the quotient torus $\T := \widehat{\T}/\BS^1_\chi$ on $X_\chi$. 
Let $\mathfrak{t}$ be the Lie algebra of $\T$, and choose a moment map $m_\chi :X_\chi \to \mathfrak{t}^\vee$.
We denote its image by $\Pol := m_\chi(X_\chi) \subset \mathfrak{t}^\vee$, the corresponding moment polytope.

\smallskip
Following \cite[Section~6.2]{Apostolov_Calderbank_2021}, we consider
deformations of the reference Sasaki structure $(N,\CD,J_N,\chi)$
obtained by varying the K\"ahler metric within $[\om_\chi]$.
Let 
\[
    \CH(X_\chi,\om_\chi)^\T
    := \left\{ \vph \in \CC^\infty(X_\chi,\BR)^\T \mid  \om_\vph := \om_\chi + \ddc_\chi \vph > 0 \right\},
\]
where $\dc_\chi$ denotes the $J_\chi$-twisted exterior derivative on
$X_\chi$. 
For $\vph \in \CH(X_\chi,\om_\chi)^\T$, set 
\[
    \eta_\vph := \eta_\chi + \pi_\chi^* \dc_\chi\vph,
    \qquad
    \CD_\vph := \ker \eta_\vph.
\]
Then $\dd \eta_\vph = \pi_\chi^* \om_\vph$.
Let $J_\varphi$ be the horizontal lift of $J_\chi$ to $\CD_\vph$ with respect to the connection $\eta_\vph$ on the principal bundle $N \to X_\chi$.
The contact vector field $\chi_\vph := \eta_{\CD_\vph}(\chi^\CD)$ defines a Sasaki structure on $(\CD_\vph, J_\vph)$ with the same Sasaki--Reeb vector field $\chi^\CD$. 

\smallskip
Next, fix $\xi \in \widehat{\mathfrak{t}} \cap \mathfrak{cr}^{\chi}_+(N,\CD,J_N)$ and let $\xi^\CD$ be its Sasaki--Reeb vector field. 
Define $\xi_\vph := \eta_{\CD_\vph}(\xi^{\CD})$, where $\eta_{\CD_\vph}: TN \to TN/\CD_\vph$ is the canonical projection. 
Since $\vph$ is $\T$-invariant, the contact form $\eta_\vph$ is invariant under $\xi^{\CD}$. 
Therefore, $\xi^{\CD}$ preserves $\CD_\vph$ and it is the contact lift of $\xi_\vph$ with respect to $\CD_\vph$. 
By \cite[Lemma~6]{Apostolov_Calderbank_2021}, $\xi_\vph$ defines a Sasaki structure on $(N,\CD_\vph,J_\vph)$. 
Its Sasaki--Reeb vector field is $\xi^{\CD}$, which is independent of $\vph$. 

\smallskip
We denote such a family of Sasaki structures by
\begin{equation*}
    \CS_\chi(\xi^{\CD})
    := \left\{ \text{Sasaki structures }(\CD_\vph,J_\vph,\xi_\vph)
    \text{ on }N\text{ induced by }
    \vph \in \CH(X_\chi,\om_\chi)^\T\right\},
\end{equation*}
which is the set of Sasaki structures on $N$ obtained by the above construction with Sasaki--Reeb vector field $\xi^{\CD}$.
Since $\xi$ commutes with $\chi$, the vector field $\xi^{\CD}$ descends to a Hamiltonian holomorphic vector field on $(X_\chi,\om_\chi)$. 
With respect to the chosen normalization of $m_\chi$, let $\ell_\xi$ be the unique positive affine-linear function on $\Pol$ satisfying 
\[
    \eta_\chi(\xi^{\CD})
    = \pi_\chi^* \left(\ell_\xi\circ m_\chi\right).
\]

\subsubsection{YTD type results}
The following results of \cite[Theorem 1]{Apostolov_Calderbank_2021} and \cite[Proposition~1]{Apostolov_Calderbank_Legendre_2021} relate extremal Sasaki structures on $N$ to weighted extremal K\"ahler metrics on $X_\chi$.

\begin{thm}
Let $\xi \in \widehat{\mathfrak{t}} \cap \mathfrak{cr}^{\chi}_+(N,\CD,J_N)$. 
A Sasaki structure $(\CD_\vph,J_\vph,\xi_\vph) \in \CS_\chi(\xi^{\CD})$ is extremal if and only if the K\"ahler metric $\om_\vph \in [\om_\chi]$ is $(v_\xi,w_\xi)$-extremal, where
\begin{equation}{\label{weight:sasa}}
    v_\xi = \ell_\xi^{-n-1},
    \qquad
    w_\xi = \ell_\xi^{-n-3}.
\end{equation}
\end{thm}

\begin{rmk}
The formulation in \cite[Theorem~1]{Apostolov_Calderbank_2021} uses weighted extremal Sasaki structures on $N$, rather than weighted extremal
K\"ahler metrics on the regular quotient $X_\chi$. 
The formulation in terms of weighted extremal metrics is established in \cite[Proposition 1]{Apostolov_Calderbank_Legendre_2021}.
\end{rmk}

\begin{rmk}
Let
\(
    D_\xi:=TN/\langle \xi^{\CD}\rangle
\).  The complex structure $J_N$ on $\mathcal D$ therefore induces a
complex structure $J_\xi$ on $D_\xi$.
Following \cite{Apostolov_Calderbank_2021}, we denote by
$\mathcal S(\xi^{\CD},J_\xi)$ the Sasaki polarization consisting of
contact forms compatible with the transverse holomorphic structure
$(\xi^{\CD},J_\xi)$, namely
\[
    \CS(\xi^{\CD},J_\xi)
    =
    \left\{\eta \in \Om^1(N) \mid \eta(\xi^{\CD})=1,\,\CL_{\xi^{\CD}}\eta = 0,\,\dd\eta\text{ is positive and }J_\xi\text{-invariant}
    \right\},
\]
where, by abuse of notation, we identify $d\eta$ with the induced
$2$-form acting on $D_\xi$.

By \cite[Corollary~2.15]{Apostolov_Calderbank_Legendre_2021},
Apostolov--Calderbank--Legendre show that the Sasaki structures in
$\CS_\chi(\xi^{\CD})$ can be  identified with a subset of
$\CS(\xi^{\CD},J_\xi)$, yielding an embedding
\[
    \CS_\chi(\xi^{\CD}) \hookrightarrow \CS(\xi^{\CD},J_\xi).
\]

Combining \cite[Theorem~1]{Apostolov_Calderbank_2021} with \cite[Proposition~3.7]{Apostolov_Calderbank_Legendre_2021}, the existence of an extremal Sasaki structure in $\CS_\chi(\xi^{\CD})$ is equivalent to the existence of an extremal Sasaki structure in $\CS(\xi^{\CD},J_\xi)$.
\end{rmk}

As a consequence of Theorem~\ref{t:A} and Corollary~\ref{c:B}, we obtain the following Yau--Tian--Donaldson correspondence characterizing the existence of an extremal Sasaki structure in $\CS_\chi(\xi^{\CD})$ in terms of relative uniform weighted K-stability of $(X_\chi,J_\chi,[\om_\chi])$.

\begin{cor}{\label{c:sasaki}}
Let $\xi \in \widehat{\mathfrak{t}} \cap \mathfrak{cr}^{\chi}_+(N,\CD,J_N)$. 
There exists an extremal Sasaki structure in $\CS_\chi(\xi^{\CD})$ if and only if $(X_\chi,J_\chi,[\om_\chi])$ is
$\T$-relatively uniformly $(v_\xi,w_\xi)$-weighted K-stable.
\end{cor}

\subsection{Relation to Collins--Sz\'ekelyhidi stability}
Collins--Sz\'ekelyhidi introduced a notion of K-stability for polarized affine cones in \cite{Collins_Szekelyhidi_2019}. 
Its relation to weighted K-stability was studied by Apostolov--Calderbank--Legendre in \cite{Apostolov_Calderbank_Legendre_2021}. 
We briefly recall this correspondence here for completeness.

\subsubsection{Stability of polarized complex cones}
Let $L_\chi$ be the polarization of $(X_\chi,J_\chi)$ satisfying $2\pi c_1(L_\chi) = [\om_\chi]$. 
The associated punctured complex cone is
\[
    \mathscr{C}_\chi^\times 
    := (L_\chi^*)^\times
    \to X_\chi,
\]
equipped with its natural complex structure. 
Here, $(L_\chi^\ast)^\times$ denotes the total space of the dual bundle $L_\chi^\ast$ minus its zero section. 
Note that $\mathscr{C}_\chi^\times$ is diffeomorphic to $N \times \BR_{>0}$.

\smallskip
For any $\vph \in \CH(X_\chi,\om_\chi)^\T$, the CR manifold $(N, \CD_\vph, J_\vph)$ can be identified with the unit circle bundle in $(L_\chi^*,h_\varphi^*)$, where $h_\varphi$ is a Hermitian metric on $L_\chi$ whose curvature is $\om_\vph$. 
In this picture, $\eta_\vph$ is the corresponding connection form and $\chi^{\CD}$ generates the fibrewise $\BS^1$-action.

\smallskip
The Sasaki--Reeb vector field $\xi^{\CD}$ on $N$ extends to a holomorphic vector field on $\mathscr{C}_\chi^\times$. 
Following \cite{Apostolov_Calderbank_Legendre_2021}, we refer to the pair $(\mathscr{C}_\chi^\times,\xi^{\CD})$ as a \textit{polarized complex cone}. 

\smallskip
Let $(\CX,\CL)$ be a smooth $\T$-equivariant polarized test configuration for $(X_\chi,L_\chi)$ in the sense of Definition~\ref{defn:TC}. 
Following \cite[Section~6.2]{Apostolov_Calderbank_Legendre_2021}, the punctured total space $\CC := (\CL^*)^\times$ defines a smooth $\T$-equivariant test configuration of the polarized complex cone $(\mathscr{C}_\chi^\times,\xi^{\CD})$ in the sense of \cite[Definition~6.4]{Apostolov_Calderbank_Legendre_2021}. 
More precisely, $\CC$ admits a flat surjective morphism to $\BP^1$ and its non-central fibres are $\T$-equivariantly isomorphic to $\mathscr{C}_\chi^\times$ as polarized complex cones and the $\BC^*$-action on $(\CX,\CL)$ induces a $\BC^*$-action on $\CC$ that covers the standard action on $\BP^1$. 

\smallskip
Passing to the associated affine cones relates this construction to the affine test configurations considered by Collins--Sz\'ekelyhidi (see \cite[Section~6.2 and Proposition~6.8]{Apostolov_Calderbank_Legendre_2021}). 
In the affine description, after choosing an equivariant embedding into $\BC^N$, a degeneration is induced by a one-parameter subgroup of $\RGL(N)$ commuting with the torus action, and its central fibre is obtained as the corresponding flat limit. 
The formulation used in \cite{Collins_Szekelyhidi_2019} imposes the additional requirement that this central fibre be normal.

\subsubsection{Comparison of the weighted Donaldson--Futaki and global Futaki invariants}
For a test configuration $\CC$ of a polarized complex cone, Apostolov--Calderbank--Legendre associate a differential-geometric invariant $\mathbf{GF}_{\xi^{\CD}}(\CC)$, called the \emph{global Futaki invariant} (\cite[Definition~6.7]{Apostolov_Calderbank_Legendre_2021}).
By \cite[Proposition~6.8]{Apostolov_Calderbank_Legendre_2021}, this invariant agrees, up to a positive dimensional constant, with the Donaldson--Futaki invariant of the corresponding affine cone in the sense of Collins--Sz\'ekelyhidi. 

\smallskip
Similarly to the K\"ahler setting, an extremal Sasaki structure has constant transverse scalar curvature if and only if its \textit{contact Futaki invariant vanishes}. 
The contact Futaki invariant was introduced in \cite[Appendix~B]{Apostolov_Calderbank_Legendre_2021} and plays a role analogous to that of the Futaki invariant in extremal K\"ahler geometry. 
Since the K-stability theory of polarized affine cones developed by Collins--Sz\'ekelyhidi is formulated in the constant scalar curvature Sasaki setting, we henceforth assume that the contact Futaki invariant vanishes in order to place ourselves in the same framework. 

\smallskip
Let $(\CX,\CL)$ be a polarized test configuration for
$(X_\chi,L_\chi)$. 
We denote by $\CC$ the associated test configuration of the
polarized complex cone $(\mathscr{C}_\chi^\times,\xi^{\CD})$, as described above. 
Then, \cite[Proposition~6.9]{Apostolov_Calderbank_Legendre_2021} yields
\[ 
    \DF^{\rel}_{v_\xi,w_\xi}(\CX,\CL)
    = \frac{1}{2\pi}
    \mathbf{GF}_{\xi^{\CD}}(\CC).
\]

\subsection{Application to conformally K\"ahler, Einstein--Maxwell geometry}
Let $X = (M,J)$ be an $n$-dimensional compact complex manifold. 
Fix a maximal real torus $\T \subset \Aut_{\red}(X)$ and a $\T$-invariant K\"ahler metric $\om$, and write $\af = [\om]$. 
A Hermitian metric $\widetilde{g}$ on $X$ is conformally K\"ahler, Einstein--Maxwell (cKEM) if there exists a positive smooth function $f$ such that $g := f^2 \widetilde{g}$ is a K\"ahler metric, 
\begin{equation}\label{killing}
    \xi:=J \grad_g(f) \text{ is Killing with respect to both } g \text{ and } \widetilde{g},
\end{equation}
and
\[
    \Scal(\widetilde{g}) = \mathrm{const}.
\]
This notion was first introduced by LeBrun in \cite{LeBrun_2010}. 
Conformally K\"ahler, Einstein--Maxwell geometry has been extensively studied from various perspectives; see e.g. \cite{LeBrun_2015, Apostolov_Calderbank_Gauduchon_2016, Koca_TonnesenFriedman_2016, LeBrun_2016, Futaki_Ono_2018, Apostolov_Maschler_2019, Futaki_Ono_2017, Lahdili_2019_b, Lahdili_2020, Apostolov_Maschler_TonnesenFriedman_2022} and the references therein.

\smallskip
We now consider the $\T$-invariant setting. 
After conjugation, we may assume that the torus generated by $\xi$ is contained in $\T$, and hence that $\xi \in \mathfrak{t}$. 
We now pay more attention to $\T$-invariant K\"ahler metrics $\homg \in \af$. 
Let $\Pol_\af$ be the (common) moment polytope associated with $\af$, where we consider the usual normalization of the moment maps of K\"ahler metrics in $\af$ with respect to the reference form $\om$.
Since $\xi = J \grad_g(f)$, there exists a constant $a \in \BR$ such that $f$ is the pullback by $m_\om$ of the positive affine-linear function $\ell_f(x) := \langle \xi, x \rangle + a$ on the moment polytope $\Pol_\af$.

\smallskip
As observed in \cite{Lahdili_2020}, $\widetilde{g}$ is cKEM if and only if
$\om$ is a $(v_f,w_f)$-extremal metric for the weights
\[
    v_f = \ell_f^{-2n+1},
    \qquad
    w_f = \ell_f^{-2n-1}.
\]
We then introduce the space of conformally K\"ahler metrics
\begin{equation*}
    \CC_{\af,\xi}^\T :=
    \left\{\ell_f(m_{\homg})^{-2}g_{\homg} \mid \homg \in \af \text{ is a $\T$-invariant K\"ahler metric}
    \right\},
\end{equation*}
consisting of Hermitian metrics conformal to $\T$-invariant K\"ahler metrics in $\af$, with conformal factor $\ell_f(m_\homg)^{-2}$.

\smallskip
Hence, Theorem \ref{t:A} and Corollary \ref{c:B} yield the following Yau--Tian--Donaldson correspondence for cKEM metrics.

\begin{cor}
The following statements are equivalent.
\begin{itemize}
    \item There exists a conformally K\"ahler, Einstein--Maxwell metric in $\CC_{\af,\xi}^\T$.
    \item $(X,\af)$ is $\T$-relatively uniformly $(v_f,w_f)$-weighted K-stable.
\end{itemize}
\end{cor}

We mention that in the toric setting the existence of a conformally K\"ahler, Einstein--Maxwell metric was related to K-polystability in \cite{Apostolov_Maschler_2019}.

\bibliographystyle{smfalpha_new}
\bibliography{biblio}
\end{document}